\documentclass[11pt,a4paper,twoside,reqno]{amsart}
\usepackage{amsfonts,amssymb,euscript,amscd}
\DeclareFontFamily{OML}{cyr}{} \DeclareFontShape{OML}{cyr}{m}{n}{
  <5> <6> <7> <8> <9> gen * wncyr
  <10> <10.95> <12> <14.4> <17.28> <20.74> <24.88> wncyr10
  }{}
\DeclareSymbolFont{rusletters}{OML}{cyr}{m}{n}
\DeclareSymbolFontAlphabet{\rusmath}{rusletters}
\DeclareMathSymbol\re{\rusmath}{rusletters}{"03}
\newtheorem{theorem}{Theorem}[section]
\newtheorem{proposition}[theorem]{Proposition}
\newtheorem{lemma}[theorem]{Lemma}
\newtheorem{corollary}[theorem]{Corollary}
\newtheorem{remark}[theorem]{Remark}
\newcommand*{\Smbl}{\mathop{\rm Smbl}\nolimits}

\newcommand*{\Orb}{\mathop{\rm Orb}\nolimits}

\newcommand*{\G}{\mathop{\rm G\!}\nolimits} 
\newcommand*{\E}{\EuScript E} 
\newcommand*{\Y}{\EuScript Y} 
\newcommand*{\A}{\EuScript A} 
\newcommand*{\K}{\EuScript K} 
\newcommand*{\R}{\mathbb R} 
\newcommand*{\g}{\mathfrak g} 
\newcommand*{\eq}{y''=a^3(x,y)y'\,^3+a^2(x,y)y'\,^2+a^1(x,y)y'+a^0(x,y)}
\newcommand{\kch}{\mathbin{\rule{5pt}{0.5pt}\rule{0.5pt}{6pt}}\,} 
\title[Differential invariants]{Differential invariants of 2-order ODEs, I}
\author[V.A. Yumaguzhin]{Valeriy A. Yumaguzhin}
\date{29 November 2007}
\address{Program Systems Institute of RAS, Pereslavl'-Zalesskiy, 152020, Russia}
\email{yuma@diffiety.botik.ru}
\keywords{2-nd order ordinary differential equation, point transformation, equivalence problem, differential invariant}
\subjclass{53A55, 53C10, 53C15, 34A30, 34A26, 34C20, 58F35}
\begin{document}
\begin{abstract} In this paper, we investigate the action of pseudogroup of all point transformations on the natural bundle of equations
  $$
    \eq\,.
  $$
We construct differential invariants of this action and solve the equivalence problem for some classes of these equations in particular for generic equations.
\end{abstract}
\maketitle
%
\section{Introduction}
This paper is devoted to differential invariants and the equivalence problem 
of ordinary differential equations of the form
\begin{equation}\label{eq}
 \eq\,.
\end{equation}

There are different approaches to construct  differential invariants of these equations, 
see R. Liouville \cite{Lvll}, S. Lie \cite{Li1,Li2}, A. Tresse \cite{Trss}, E. Cartan \cite{Crtn}, G. Thomsen \cite{Thmsn}, and R.B. Gardner \cite{Grdnr}.
 
In \cite{Yum}, we presented an approach to this problem differing from above mentioned ones. In this paper, we state in detail this approach, construct tensor and scalar differential invariants in this way, and solve the equivalence problem for some classes of equations \eqref{eq}, in particular, for generic equations.

Briefly, our approach is as follows.
Every equation $\E$ of form \eqref{eq} can be considered as a geometric structure. To this end, we identify the equation $\E$ with the section
$$
  S_{\E}:(x,y)\mapsto(\,x,y,\,a^0(x,y),\,a^1(x,y),\,a^2(x,y),\,a^3(x,y)\,)
$$
of the product bundle $\pi:\R^2\times\R^4\longrightarrow\R^2$. Thus the set of all equations \eqref{eq} is identified with the set of all sections of $\pi$. It is well known, see \cite{Arnld}, that every point transformation of variables $x$ and $y$ transforms every equation \eqref{eq} to equation of the same form\footnote{2nd--order ODEs admit contact transformations. It is known, see \cite{Chrn}, that every two 2nd--order ODEs are locally contact equivalent. By this reason, we investigate the equivalence problem of  equations \eqref{eq} w.r.t. point transformations.}. It follows that every point transformation $f$ of the base of $\pi$ generates the transformation of sections of $\pi$. This means that $f$ can be lifted in the natural way to the diffeomorphism $f^{(0)}$ of the total space of $\pi$. Thus the bundle $\pi$ of equations \eqref{eq} is a natural bundle. Therefore equation \eqref{eq} considered as a section of $\pi$, is a geometric structure, see \cite{ALV}. 
By $\pi_k:J^k\pi\to\R^2$ denote the bundle of $k$--jets of sections of $\pi$, $k=1,2,\ldots$. Every lifted diffeomorphism $f^{(0)}$ is lifted in the natural way to the diffeomorphism $f^{(k)}$ of $J^k\pi$. 
The lifting of point transformations generates the natural lifting of every vector field $X$ in the base of $\pi$ to the vector field $X^{(k)}$ in $J^k\pi$. Suppose a lifted vector field $X^{(k)}$ passes through a point $\theta_k\in J^k\pi$. Then the value $X^{(k)}_{\theta_k}$ of this field at $\theta_k$ is defined by the $k+2$--jet $j_p^{k+2}X$ of the field $X$ at the point $p=\pi_k(\theta_k)$.
Let $\theta_{k+1}\in J^{k+1}\pi$. Then there exists a section $S$ of $\pi$ such that $\theta_{k+1}=j_p^{k+1}S$, where $p=\pi_{k+1}(\theta_{k+1})$. The section $S$ generates the section $j_kS$ of the bundle $\pi_k$ by the formula $j_kS:p\mapsto j_p^kS$. It is clear that $\theta_{k+1}$ is identified with the tangent space to the image of $j_kS$ at the point $\theta_k=j_p^kS$. We denote this tangent space by $\K_{\theta_{k+1}}$. Now we can introduce the following vector space of $k+2$--jets at $p$ of vector fields in the base passing through $p$\,:  
$$
  \A_{\theta_{k+1}}=\{\,j_p^{k+2}X\,|\,X^{(k)}_{\theta_k}\in\K_{\theta_{k+1}}\,\}\,.
$$
The spaces $\A_{\theta_{k+1}}$, $k=0,1,2,\ldots$, possess nontrivial properties. These properties allow us to construct in the natural way some geometric objects $\omega_{\theta_{k+1}}$ on the tangent space to the base at the point $p=\pi_{k+1}(\theta_{k+1})$. As a result, we obtain fields of these objects on $J^{k+1}\pi$
$$
  \theta_{k+1}\longmapsto\omega_{\theta_{k+1}}\,.
$$        
These fields are differential invariants of the considered equations w.r.t. point transformations. 

The pseudogroup of all point transformations of the base acts by the lifted diffeomorpisms on every $J^k\pi$. As a result, every $J^k\pi$ is divided into orbits of this action. The bundles $J^0\pi$ and $J^1\pi$ are orbits of this action. The bundle $J^2\pi$ is the union of two orbits: $\Orb_2^0$ and $\Orb_2^2$. First one is an orbit of codimension 0, the second one has codimension 2 and consists of 2--jets of sections $S_{\E}$ such that the equation $\E$ can be reduced to the linear form by a point transformation, see \cite{GYum, Yum}. 
The bundle $J^3\pi$ is the union of four orbits: an orbit $\Orb_3^0$ of codimension $0$, an orbit $\Orb_3^1$ of codimension 1, an orbit $\Orb_3^2$ of codimension 2, and the orbit of codimension 6 that is the inverse image of $\Orb_2^2$ over the natural projection $J^3\pi\to J^2\pi$. 

In this paper, we construct differential invariants and solve the equivalence problem for equations $\E$ satisfying the condition $j_3S_{\E}\subset\Orb_3^0$.
\smallskip

All manifolds and maps are smooth in this work. By $j_p^kf$ denote the $k$--jet of the map $f$ at the point $p$, $k=0,1,2,\ldots,\infty$, by $\R$ denote the field of real numbers, and by $\R^n$ denote the $n$--dimensional arithmetic space. We assume summation over repeated indexes in all formulas.

\section{The bundle of equations}
%
\subsection{Liftings of point transformations}
Consider the product bundle 
$$
  \pi:\R^2\times\R^4\longrightarrow\R^2,\quad\pi:(x^1,x^2,u^1,\ldots,u^4)\mapsto (x^1,x^2),
$$ 
where $x^1,x^2$ are the standard coordinates on the base of $\pi$ and $u^1$, $u^2$, $u^3$, $u^4$ are 
the standard coordinates on the fiber of $\pi$.

Let $\E$ be an arbitrary equation \eqref{eq}. We identify $\E$ with the section $S_{\E}$ of $\pi$ defined by the formula
$$
  S_{\E}(p)= \bigl(\,p,\,a^0(p),\,a^1(p),\,a^2(p),\,a^3(p)\,\bigr),
$$
where $p=(x^1,x^2)$. Clearly, this identification is a bijection between the set of all equations \eqref{eq} and the set of all sections of $\pi$.

Recall that a {\it point transformation of} $\R^2$ is a diffeomorphism of some open subset of $\R^2$ to $\R^2$. 

Every point transformation $f$ of $\R^2$ generates the transformation of $\E$ to the equation $\tilde\E$ of the same form, see \cite{Arnld}, 
$$  
  \tilde y''=\tilde a^3(\tilde x,\tilde y)\tilde y'\,^3
  +\tilde a^2(\tilde x,\tilde y)\tilde y'\,^2
  +\tilde a^1(\tilde x,\tilde y)\tilde y'
  +\tilde a^0(\tilde x,\tilde y)\,.
$$
The coefficients of $\tilde\E$ are expressed in terms of the coefficients of $\E$ and the $2$-jets of the inverse
transformation $f^{-1}$:
\begin{equation}\label{CffTrnsfrm}
 \tilde a^i(\tilde p)=\Phi^i\bigl(\,a^0(f^{-1}(\tilde p)),\ldots,a^3(f^{-1}(\tilde p)),\,
 j^2_{\tilde p}f^{-1}\,\bigr),\quad i=0,1,2,3.
\end{equation}
It follows that the equations
$$
  \tilde p=f(p),\quad \tilde u^i=\Phi^i\bigl(\,u^1,\ldots,u^4,\,j^2_{f(p)}f^{-1}\,\bigr),\quad i=1,2,3,4.
$$
define the diffeomorphism $f^{(0)}$ of the total space of $\pi$. It is easy to see that if $U$ is domain of $f$, then $f^{(0)}$ is defined on $\pi^{-1}(U')$, where $U'$ is the everywhere dense open subset of $U$. This diffeomorphism $f^{(0)}$ is called the {\it lifting of $f$ to the bundle $\pi$}. Obviously, the diagram
$$
  \begin{CD}
    E        @>f^{(0)}>> E\\
    @V\pi VV             @VV\pi V\\
    \R       @>>f>       \R
  \end{CD}
$$
is commutative (in the domain $\pi^{-1}(U')$ of $f^{(0)}$). 

Now equations \eqref{CffTrnsfrm} is represented in the terms of the transformation of the corresponding sections in the following way
$$
    S_{\tilde\E}=f^{(0)}\circ S_{\E}\circ f^{-1}.
$$

By $j_p^kS$ denote the the $k$--jet at $p$ of the section $S$ of $\pi$, $k=0,1,2,\ldots,\infty$. By
$$
  \pi_k:J^k\pi\longrightarrow\R^2,\quad\pi_k:j_p^kS\mapsto p\,,
$$
denote the bundle of all $k$--jets of sections of $\pi$. By $x^1$, $x^2$, $u^i_{\sigma}$, $i=1,\ldots,4$, $0\leq |\sigma|\leq k$, we denote the standard coordinates on $J^k\pi$, here $\sigma$ is the multi-index 
$\{j_1\ldots j_r\}$, $|\sigma|=r$, $j_1,\ldots,j_r=1,2$. By $\sigma j$ we denote the multi-index
$\{j_1\ldots j_r j\}$.
The natural projection
$$
  \pi_{k,\;r}:J^k\pi\longrightarrow J^r\pi\,,\quad\infty\geq k>r,
$$
is defined by $\pi_{k,\;r}(\,j_p^kS\,)=j_p^rS$. By $J_p^k\pi$ denote the fiber of the bundle $\pi_k$ over the point $p$, that is $J_p^k\pi=\pi_k^{-1}(p)$.
Every section $S$ of $\pi$ generates the section $j_kS$ of the
bundle $\pi_k$ by the formula
$$
  j_kS: p\;\mapsto\;j_p^kS.
$$

Every point transformation $f$ of the base of $\pi$ is lifted to the diffeomorphism $f^{(k)}$ of $J^k\pi$ by the formula
\begin{equation}\label{LftTr}
  f^{(k)}(\,j^k_pS\,)=j^k_{f(p)}\bigl(\,f^{(0)}\circ S\circ f^{-1}\,\bigr)\,.
\end{equation}
This diffeomorphism $f^{(k)}$ is called the {\it lifting of $f$ to the bundle $\pi_k$}
Obviously, for any $\infty\geq l>m$, the diagram
$$
  \begin{CD}
    J^l\pi         @>f^{(l)}>> J^l\pi\\
    @V\pi_{l,\;m}VV             @VV\pi_{l,\;m} V\\
    J^m\pi         @>>f^{(m)}> J^m\pi
  \end{CD}
$$
is commutative. Suppose $f$ and $g$ are point transformations of the base of $\pi$, then obviously,
$$
  (f\circ g)^{(k)}=f^{(k)}\circ g^{(k)},\;\;k=0,1,\ldots
$$
By $\Gamma$ we denote the pseudogroup of all point transformations of the base of $\pi$.  The pseudogroup $\Gamma$ acts on every $J^k\pi$ by the lifted transformations.

\subsection{Liftings of vector fields}
Let $X$ be a vector field in the base of $\pi$ and let $f_t$ be its flow. Then the flow $f_t^{(k)}$ in $J^k\pi$ defines the vector field $X^{(k)}$ in $J^k\pi$, which is called {\it the lifting of $X$ to $J^k\pi$}.

It follows from the definition:
$$
 (\pi_{l,\;m})_*\bigl(\,X^{(l)}\,\bigr)=X^{(m)}\,,\;\;\infty\geq l>m\geq -1\,,
$$
where $\pi_{l,\;-1}=\pi_l$ and $X^{(-1)}=X$, and
$$
 (f^{(k)})_*(X^{(k)})=(\,f_*(X)\,)^{(k)},\quad k=0,1,\ldots\,,
$$
where $f$ is an arbitrary point transformation of the base of $\pi$. 
\begin{proposition}\label{LieAlgIsm} The map $X\longmapsto X^{(k)}$ is a Lie algebra homomorphism of the algebra of all 
  vector fields in the base of $\pi$ to the algebra of all vector fields in $J^k\pi$.
\end{proposition} 
\begin{proof} See subsection \ref{SctnLieAlgIsm} of Appendix.
\end{proof} 

Recall the formulas describing lifted vector fields in the terms of the standard coordinates of $J^k\pi$, see \cite{KLV}, \cite{KV}.
Let $S$ be a section of $\pi$ defined in the domain of $X$, $p$ be a point of this domain, and $\theta_1=j^1_pS$. Then the vector-function $\psi_X$ defined by the formula
$$
  \psi_X(\theta_1)
  =\begin{pmatrix}
     \psi^1_X(\theta_1)\\
     \cdots\\
     \psi^4_X(\theta_1)\,.
  \end{pmatrix}
  =\frac{d}{dt}(\,f_t^{(0)}\circ S\circ
  f_t^{-1}\,)\Bigr|_{t=0}(p)
$$
is the {\it deformation velocity} of the section $S$ at the point $p$ under the action of the flow of $X$. Suppose
$$
  X=X^1\frac{\partial}{\partial x^1}
  +X^2\frac{\partial}{\partial x^2}\quad\text{and}\quad
  \theta_1=(p, u^i, u^i_j).
$$
Then 
\begin{equation}\label{DfrmtnVlst1}
  \psi_X(\theta_1)=
  \begin{pmatrix}
    -u^1_1X^1-u^1_2X^2\\
    - 2u^1X^1_1 + u^1X^2_2 - u^2X^2_1\\
    + X^2_{11}\vspace{.1in}\\
    -u^2_1X^1-u^2_2X^2\\
    - 3u^1X^1_2 - u^2X^1_1 - 2u^3X^2_1\\
    - X^1_{11} + 2X^2_{12}\vspace{.1in}\\
    -u^3_1X^1-u^3_2X^2\\
    - 2u^2X^1_2 - u^3X^2_2 - 3u^4X^2_1\\
    - 2X^1_{12} + X^2_{22}\vspace{.1in}\\
    -u^4_1X^1-u^4_2X^2\\
    -  u^3X^1_2 + u^4X^1_1 - 2u^4X^2_2\\
    - X^1_{22}
  \end{pmatrix},
\end{equation}
where $X^i_j=\displaystyle\frac{\partial X^i}{\partial x^j}(p)$ and $X^i_{j_1j_2}=\displaystyle\frac{\partial^2X^i}{\partial x^{j_1}\partial x^{j_2}}(p)$.
The vector field $X^{(\infty)}$ is described by the formula
$$
  X^{(\infty)}=X^1D_1+X^2D_2+\re_{\psi_X}, 
$$
where
\begin{equation}\label{Dj}
  D_j=\frac{\partial}{\partial x^j}+\sum_{|\sigma|=0}^{\infty}\sum_{i=1}^{4}u^i_{\sigma j}
  \frac{\partial}{\partial u^i_{\sigma}}\,,\quad j=1,2\,,
\end{equation}  
is the operator of total derivative w.r.t. $x^j$ and  
\begin{equation}\label{Re}
  \re_{\psi_X}=\sum_{|\sigma|=0}^{\infty}\sum_{i=1}^{4}D_{\sigma}\bigl(\,\psi^i_X\,\bigr)
  \frac{\partial}{\partial u^i_{\sigma}},\quad
  D_{\sigma} =D_{j_1}\circ\ldots\circ D_{j_r}.
\end{equation}  
The vector field $X^{(k)}$ is described by the formula
\begin{equation}\label{kPrlng}
  X^{(k)}=(\pi_{\infty,\;k})_*(X^{(\infty)})=X^1D_1^k+X^2D_2^k+\re_{\psi_X}^k,
\end{equation}
where
\begin{gather}
  D_j^k=\frac{\partial}{\partial x^j}
  +\sum_{|\sigma|=0}^k\sum_{i=1}^{4}
  u^i_{\sigma j}
  \frac{\partial}{\partial u^i_{\sigma}}\,,\quad j=1,2\,,\label{D(k)j}\\ 
  \re_{\psi_X}^k=\sum_{|\sigma|=0}^k
  \sum_{i=1}^{4}D_{\sigma}\bigl(\,\psi^i_X\,\bigr)\frac{\partial}{\partial u^i_{\sigma}}\,.\notag
\end{gather}

The following important statement is obvious.
\begin{proposition}\label{VlLftVctFld} Let $\theta_k\in J^k\pi$, $p=\pi_k(\theta_k)$, and $X^{(k)}$ be a lifted vector field passing through $\theta_k$. Then the value $X^{(k)}_{\theta_k}$ of $X^{(k)}$ at the point $\theta_k$ is defined by the $(2+k)$--jet $j_p^{2+k}X$ of the vector field $X$ at the point $p$.
\end{proposition}

\section{Isotropy algebras and orbits}
%
\subsection{Jets of vector fields} 
In this subsection, we recall necessary notions concerning jets of vector fields, prolongations of subspaces, and Spencer's complexes 
, see \cite{BrnshtnRznfld}, and \cite{GllmnStrnbrg1}.

By $W_p$ we denote the Lie algebra of $\infty$--jets at $p\in\R^2$ of all vector fields defined in $\R^2$ in neighborhoods of $p$. Recall that the structure of Lie algebra on $W_p$ is defined by the following formulas
\begin{gather*}
  \lambda j^{\infty}_pX=j^{\infty}_p(\lambda X)\,,\quad
  j^{\infty}_pX+j^{\infty}_pY=j^{\infty}_p(X+Y)\,,\\
  \bigl[\,j^{\infty}_pX,j^{\infty}_pY\,\bigr]=j^{\infty}_p[X,Y]\quad
  \forall\,\lambda\in\R\,,\;\forall\,j^{\infty}_pX,\,j^{\infty}_pY\in W_p\,.
\end{gather*}
By $L_p^k\,,\;k=-1,0,1,2,\ldots$, we denote the subalgebra of $W_p$ defined by 
$$
  L_p^k=\bigl\{\,j^{\infty}_pX\in W_n\,\bigl|\,j^k_pX=0\,\bigr\}\,,\;k\geq 0\,,
  \quad L_p^{-1}=W_p\,.
$$ 
Obviously, $W_p/L_p^k$ is the vector space of all $k$-jets at $p$ of all vector fields passing through $p$. In particular, $W_p/L_p^0$ is the tangent space $T_p$ to $\R^2$ at $p$. We have the natural filtration 
\begin{equation}\label{Fltrtn}
  W_p=L_p^{-1}\supset L_p^0\supset L_p^1\supset\ldots\supset L_p^k\supset L_p^{k+1}\supset\ldots\,.
\end{equation}
By $\rho_{i,j}$, $i>j\geq 0$, we denote the natural projection 
$$
  \rho_{i,j}:W_p/L_p^i\to W_p/L_p^j\,,\quad\rho_{i,j}:j_p^iX\mapsto j_p^jX.
$$ 
It is easy to prove that 
$$
  [\,L_p^i\,,\;L_p^j\,]\,=\,L_p^{i+j}\,,\quad  i,j=-1,0,1,2,\ldots\,.
$$ 
It follows that the bracket operation on $W_p$ generates the Lie algebra structure on the vector space $L_p^0/L_p^k$
\begin{equation}\label{brkt0}
  [\,\cdot\,,\,\cdot\,]:\,L_p^0/L_p^k\times L_p^0/L_p^k\to L_p^0/L_p^k 
\end{equation}  
and generates the following maps\,: 
\begin{align}
  [\,\cdot\,,\,\cdot\,]:\,&W_p/L_p^k\times W_p/L_p^k\to W_p/L_p^{k-1}\,,\label{brckt1}\\
  [\,\cdot\,,\,\cdot\,]:\,&T_p\times L_p^k/L_p^{k+1}\to L_p^{k-1}/L_p^k\,.\label{brckt2}
\end{align}
The last map generates the isomorphism 
\begin{equation}\label{Ismrph}
  L_p^k/L_p^{k+1}\cong T_p\otimes S^k(T_p^*)\,.
\end{equation}
Let $g^k$ be a subspace of $L_p^{k-1}/L_p^k$. The subspace $(g^k)^{(1)}\subset L_p^k/L_p^{k+1}$ is defined by 
$$
  (g^k)^{(1)}=\bigl\{\,X\in L_p^k/L_p^{k+1}\,\bigl|\,[\,v\,,\,X\,]\in g^k\;\;\forall\;v\in T_p\,\bigr\}
$$ 
and is called {\it the 1-st prolongation of $g^k$}.
Assume that the sequence of subspaces $g^1\,,\;g^2\,,\;\ldots\,,\;g^i\,,\;\ldots$ satisfies to the property 
$[\,T_p\,,\;g^{i+1}\,]\subset g^i$. Then for every $g^i$, we have the Spencer's complex
\begin{equation}\label{SpncrCmplx}
  0\to g^i\xrightarrow{\partial_{i,0}}g^{i-1}\otimes T_p^*
  \xrightarrow{\partial_{i-1,1}}g^{i-2}\otimes\wedge^2 T_p^*
  \xrightarrow{\partial_{i-2,2}} 0\,,
\end{equation}
where the operators $\partial_{k,l} : g^k\otimes\wedge^l T_p^*\to
g^{k-1}\otimes\wedge^{l+1} T_p^*$ are defined in the following
way: every element $\xi\in g^k\otimes\wedge^l T_p^*$ can be
considered as an exterior form on $T_p$ with values in $g^k$, then
\begin{equation}\label{SpncrOpr}
  (\,\partial_{k,l}(\xi)\,)(v_1,\ldots,v_{l+1})
  =\sum_{i=1}^{l+1}(-1)^{i+1}[\,v_i\,,\;\xi(v_1,\ldots,\hat
  v_i,\ldots,v_{l+1})\,]\,.
\end{equation}

\subsection{Isotropy algebras}\label{SctnIstrAlg}
Let $\theta_k\in J^k\pi$ and $p=\pi_k(\theta_k)$. By $G_{\theta_k}$ we denote the {\it isotropy group} of $\theta_k$, 
$$
  G_{\theta_k}=\bigl\{\;j^{2+k}_pf\;\bigl|\;f\in\Gamma\,,\;
  f^{(k)}(\theta_k)=\theta_k\;\bigr\}.
$$
By $\g_{\theta_k}$ we denote the Lie algebra of $G_{\theta_k}$. It can be considered as a subalgebra of $L^0_p/L^{2+k}_p$,
\begin{equation}\label{IstrpAlgk}
  \g_{\theta_k}=\bigl\{\;j^{2+k}_pX\in L^0_p/L_p^{2+k}\;\bigl|\; X^{(k)}_{\theta_k}=0\;\bigr\}
\end{equation}
The algebra $\g_{\theta_k}$ is called the {\it isotropy algebra} of $\theta_k$.
From this definition and \eqref{kPrlng}, we get
\begin{proposition}\label{IstrpAlgkPr} Let $j^{2+k}_pX=(p,0,X^i_j,\ldots,X^i_{j_1\ldots j_{2+k}})$ in the standard coordinates. Then $j^{2+k}_pX\in\g_{\theta_k}$ iff $(0,X^i_j,\ldots,X^i_{j_1\ldots j_{2+k}})$ is a solution of the system of linear homogeneous algebraic equations
$$
  \bigl(\,D_{\sigma}(\,\psi^i_{X}\,)\,\bigr)(\theta_k)=0\,,\quad i=1,2,3,4\,,\;\; 0\leq|\sigma|\leq k\,.
$$
\end{proposition}
The natural filtration \eqref{Fltrtn} generates the natural filtration of $\g_{\theta_k}$ 
$$
  \g_{\theta_k}=\g^1_{\theta_k}\supset\g^2_{\theta_k}\supset\ldots\supset\g^{2+k}_{\theta_k}\,,
$$
where 
$$
  \g^i_{\theta_k}=\g_{\theta_k}\cap L_p^{i-1}/L_p^{2+k},\quad i=1,2,\ldots,2+k.
$$ 
This filtration generates the graduate space 
$$
  \G\g_{\theta_k}= g_{\theta_k}^1\oplus g_{\theta_k}^2\oplus\ldots\oplus g_{\theta_k}^{k+2},
$$
where
$$
  g^i_{\theta_k}=\g^i_{\theta_k}/\g^{i+1}_{\theta_k},\; i=1,2,\ldots,k+1,\quad 
  g^{k+2}_{\theta_k}=\g^{k+2}_{\theta_k}.
$$ 
\subsubsection{Algebras $\g_{\theta_0}$}
Let $\theta_0\in J^0\pi$ and $p=\pi(\theta_0)$. The algebra $\g_{\theta_0}$ is a subalgebra of $L_p^0/L_p^2$. 
Suppose $j^2_pX=(p,0,X^i_j,X^i_{j_1j_2})\in L_p^0/L_p^2$ and $\theta_0=(p, u^1,\ldots, u^4)$ in the standard coordinates. Then from Proposition \ref{IstrpAlgkPr}, we get that  $\g_{\theta_0}$ is described by the system
\begin{equation}\label{IstrAlg0}
  \begin{aligned}
    &- 2u^1X^1_1 + u^1X^2_2 - u^2X^2_1 + X^2_{11}=0\\
    &- 3u^1X^1_2 - u^2X^1_1 - 2u^3X^2_1- X^1_{11} + 2X^2_{12}=0\\
    &- 2u^2X^1_2 - u^3X^2_2 - 3u^4X^2_1- 2X^1_{12} + X^2_{22}=0\\
    &-  u^3X^1_2 + u^4X^1_1 - 2u^4X^2_2- X^1_{22}=0\,.
  \end{aligned}
\end{equation} 
From this system, we obtain the natural filtration of $\g_{\theta_0}$ and the corresponding graduate space $\G\g_{\theta_0}$:
\begin{equation}\label{IstrAlg01}
  \g_{\theta_0}\supset g^2,\quad\G\g_{\theta_0}=L^0_p/L^1_p\oplus g^2\,,
\end{equation}
where  $g^2=g_{\theta_0}^2$ is independent of the point $\theta_0$ and is defined by the system
\begin{equation}\label{g2}
  X^2_{11}=0,\quad X^1_{11} - 2X^2_{12}=0,\quad 2X^1_{12} - X^2_{22}=0,\quad X^1_{22}=0\,.
\end{equation}
From the last system, we get 
\begin{align}
  \dim g^2&=2\label{Dim_g2}\,,\\
  (g^2)^{(1)}&=\{0\}\label{FstPrlng}\,.
\end{align}
Taking into account isomorphism \eqref{Ismrph}, we obtain that the tensors 
\begin{equation}\label{Gnrtrs_g2}
  \begin{aligned}
    e_1&=2\frac{\partial}{\partial x^1}\otimes (dx^1\odot dx^1)
      +\frac{\partial}{\partial x^2}\otimes (dx^1\odot dx^2)\,,\\
    e_2&=2\frac{\partial}{\partial x^2}\otimes (dx^2\odot dx^2)
     +\frac{\partial}{\partial x^1}\otimes (dx^1\odot dx^2)
  \end{aligned}
\end{equation}
form the base of the vector space $g^2$.

\subsubsection{Algebras $\g_{\theta_1}$}
Let $\theta_1\in J^1\pi$, $\theta_0=\pi_{1,0}(\theta_1)$, and $p=\pi_1(\theta_1)$. The algebra $\g_{\theta_1}$ is a subalgebra of $L_p^0/L_p^3$. It follows from Proposition \ref{IstrpAlgkPr} that  $\g_{\theta_1}$ is described by the system of linear homogeneous algebraic equations
\begin{gather*}
    \psi^i_{X}(\theta_0)=0\,,\quad D_1(\,\psi^i_{X}\,)(\theta_1)=0\,,\quad            
    D_2(\,\psi^i_{X}\,)(\theta_1)=0\,,\\
     i=1,2,3,4\,.
\end{gather*}
From this system, we obtain the natural filtration of $\g_{\theta_1}$ and the corresponding graduate space $\G\g_{\theta_1}$:
$$
  \g_{\theta_1}\supset\g_{\theta_1}^2\supset\{0\}\,,\quad\G\g_{\theta_1}=L^0_p/L^1_p\oplus g^2\oplus\{0\}\,,
$$
Thus the projection
\begin{equation}\label{FstPrlng1}
  \rho_{3,2}\bigl|_{\g_{\theta_1}}:\g_{\theta_1}\longrightarrow\g_{\theta_0}
\end{equation}
is an isomorphism.

\subsubsection{Algebras $\g_{\theta_2}$}
Let $\theta_2\in J^2\pi$ and $p=\pi_2(\theta_2)$. The algebra $\g_{\theta_2}$ is a subalgebra of $L_p^0/L_p^4$. 
Suppose $j^4_pX=(p,0,X^i_j,\ldots,X^i_{j_1\ldots j_4})\in L_p^0/L_p^4$ and $\theta_2=(p, u^i, u^i_j, u^i_{j_1j_2})$ in the standard coordinates. Applying computer algebra, we reduce the system of equations describing the algebra $\g_{\theta_2}$, see Proposition \ref{IstrpAlgkPr}, to a step-form. As a result, we obtain the natural filtration of   $\g_{\theta_2}$ and the corresponding graduate space $\G\g_{\theta_2}$:
\begin{equation}\label{GrIstrpAlg2}
  \g_{\theta_2}\supset\g_{\theta_2}^2\supset\{0\}\supset\{0\}\,,\quad
  \G\g_{\theta_2}=g^1_{\theta_2}\oplus g^2\oplus\{0\}\oplus\{0\}\,,
\end{equation}
where subalgebra $g^1_{\theta_2}\subset L_p^0/L_p^1$ is defined by the system of equations
\begin{equation}\label{g1}
 \begin{aligned}
  2F^1\cdot X^1_1+F^2\cdot X^2_1+F^1\cdot X^2_2 &=0\\
  F^2\cdot X^1_1+F^1\cdot X^1_2+2F^2\cdot X^2_2 &=0\,,
 \end{aligned}
\end{equation}
and
\begin{equation}\label{F1F2}
 \begin{aligned}
  F^1 &= 3u^1_{22}-2u^2_{12}+u^3_{11}\\
      &\phantom{=}+3u^4u^1_1-3u^3u^1_2+2u^2u^2_2-u^2u^3_1-3u^1u^3_2+6u^1u^4_1\,,\\
  F^2 &= u^2_{22}-2u^3_{12}+3u^4_{11}\\
      &\phantom{=} -3u^1u^4_2+3u^2u^4_1-2u^3u^3_1+u^3u^2_2+3u^4u^2_1-6u^4u^1_2\,.
 \end{aligned}
\end{equation}
From \eqref{GrIstrpAlg2} and \eqref{g1}, we get 
\begin{proposition}
 \begin{enumerate} 
  \item $\dim\g_{\theta_2}=4$ iff $\bigl(\,F^1(\theta_2),\,F^2(\theta_2)\,\bigr)\neq 0$.
  \item $\dim\g_{\theta_2}=6$ iff $F^1(\theta_2)=0$ and $F^2(\theta_2)=0$.
 \end{enumerate} 
\end{proposition}
The conditions $F^1=0$ and $F^2=0$ can be considered as conditions for the coefficients of equation \eqref{eq}. 
The following statement is well-known, see \cite{Lvll}, \cite{Trss}, \cite{Crtn}, \cite{Thmsn}, \cite{GrssmThmpsnWlkns}, \cite{GYum}, and \cite{Yum}. 
\begin{proposition}\label{LnrstnCndtnPr} The conditions $F^1=0$ and $F^2=0$ are necessary and sufficient to exists a point transformation reducing equation \eqref{eq} to the linear form.
\end{proposition}
From this proposition, we get
\begin{corollary} Let $\E$ be equation \eqref{eq}. Then it can be reduced to the linear form by a point transformation   
 iff the isotropy algebra of every 2--jet of the section $S_{\E}$ is 6--dimensional.
\end{corollary}
From \eqref{g1} we get the system of equations defining the 1-st prolongation $(g_{\theta_2}^1)^{(1)}$ of the algebra $g_{\theta_2}^1$
$$
 \begin{aligned}
     &2F^1\cdot X^1_{11}+0\cdot X^1_{12}+0\cdot X^1_{22}+F^2\cdot X^2_{11}+F^1\cdot X^2_{12}+0\cdot X^2_{22}=0,\\
     &0\cdot X^1_{11}+2F^1\cdot X^1_{12}+0\cdot X^1_{22}+0\cdot X^2_{11}+F^2\cdot X^2_{12}+F^1\cdot X^2_{22}=0,\\
     &F^2\cdot X^1_{11}+F^1\cdot X^1_{12}+0\cdot X^1_{22}+0\cdot X^2_{11}+2F^2\cdot X^2_{12}+0\cdot X^2_{22}=0,\\
     &0\cdot X^1_{11}+F^2\cdot X^1_{12}+F^1\cdot X^1_{22}+0\cdot X^2_{11}+0\cdot X^2_{12}+2F^2\cdot X^2_{22}=0.
 \end{aligned}
$$
It is easy to prove now that
\begin{equation}\label{Dim_g1}
  \dim (g_{\theta_2}^1)^{(1)}=2.
\end{equation}

\subsubsection{Algebras $\g_{\theta_3}$}
Let $\theta_3\in J^3\pi$, $p=\pi_3(\theta_3)$, and $\theta_3=(p, u^i, u^i_j,\ldots, u^i_{j_1j_2j_3})$ in the standard coordinates. Applying computer algebra, we reduce the system of equations describing the algebra $\g_{\theta_3}$, see Proposition \ref{IstrpAlgkPr}, to a step-form. From the obtained system, we get
\begin{proposition}\label{AlgIstr3} $\dim\g_{\theta_3}=0$ iff $F^3(\theta_3)\neq 0$, where 
 \begin{equation}\label{F3}
  \begin{aligned}
   F^3&=F^2(F^1D_1F^2-F^2D_1F^1)-F^1(F^1D_2F^2-F^2D_2F^1)\\
      &\phantom{=}+(F^1)^3u^4-(F^1)^2F^2u^3+F^1(F^2)^2u^2-(F^2)^3u^1\,.
  \end{aligned}
 \end{equation}
\end{proposition}
The function $F^3$ is a coefficient of some differential invariant of the action $\Gamma$ on $J^3\pi$. Bellow, we will construct this invariant. First, it was obtained in a different way by R. Liouville in \cite{Lvll}. 

\subsection{Orbits} 
In this section we describe some orbits of the actions of the pseu\-do\-group $\Gamma$ on the bundles $J^k\pi$, $k=0,1,2,3$.

By $\Orb(\theta_k)$ we denote the orbit of the action of $\Gamma$ on $J^k\pi$ passing through $\theta_k\in J^k\pi$. It is clear that $\Gamma$ acts transitively on the base of $\pi$. Hence $\Orb(\theta_k)$ can be  reconstructed by the intersection $\Orb(\theta_k)\cap J_p^k\pi$, where $J_p^k\pi$ is the fiber of $\pi_k$ over an arbitrary point $p$ of the base. Let $\Gamma_p$ be the subgroup of $\Gamma$ consisting of all transformations preserving $p$. The subgroup $\Gamma_p$ acts on the fiber $J_p^k\pi$ and $\Orb(\theta_k)\cap J_p^k\pi$ is an orbit of this action. Taking into account the previous descriptions of the algebras $\g_{\theta_0}$, $\g_{\theta_1}$, $\g_{\theta_2}$, and $\g_{\theta_3}$, we can prove now the following theorem
\begin{theorem}\label{OrbtThrm}
 \begin{enumerate}
  \item $J^k\pi$, $k=0,1$, is an orbit of the action of $\Gamma$,
  \item $J^2\pi$ is the union of two orbits of the action of $\Gamma$, $\Orb_2^0$ and $\Orb_2^2$.
   \begin{enumerate} 
    \item $\Orb_2^0$ is a generic orbit, which is described by the inequality
     $$
      (\,F^1,\;F^2\,)\neq 0\,,
     $$  
     where $F^1$ and $F^2$ are defined by \eqref{F1F2}.
      \item $\Orb_2^2$ is a degenerate orbit of codimension $2$, which is described as a submanifold of $J^2\pi$ by the            equations:
       $$
        F^1=0,\quad F^2=0\,.
       $$
   \end{enumerate}
  \item $J^3\pi$ is a union of some orbits of the action of $\Gamma$. One of these orbits $\Orb_3^0$ is a generic orbit,    which is described by the inequality
   $$
    F^3\neq 0\,,
   $$
   where $F^3$ is defined by \eqref{F3}.
 \end{enumerate}
\end{theorem} 

\section{Spaces ${\bf \A_{\theta_{k+1}}}$}
In this section, we introduce a vector space $\A_{\theta_{k+1}}$ which is a basic notion of our approach to construct differential invariants.

Let $\theta_{k+1}\in J^{k+1}\pi$, $p=\pi_{k+1}(\theta_{k+1})$, and $S$ be a section of $\pi$ such that $j^{k+1}_pS=\theta_{k+1}$. Then $\theta_{k+1}$ is identified with the tangent space to the image of the section $j_kS$ at the point $\theta_k=j_p^kS$. We denote this tangent space by $\K_{\theta_{k+1}}$. Obviously, in the standard coordinates,
$$
  \K_{\theta_{k+1}}=\langle\,D_1^k\bigl|_{\theta_k},\,D_2^k\bigl|_{\theta_k}\,\rangle,
$$   
where  $D_1^k$ and $D_2^k$ are the operators of total derivatives w.r.t. $x^1$ and $x^2$ respectively, see formula \eqref{D(k)j}. 

Now we can introduce the vector space $\A_{\theta_{k+1}}$,
\begin{equation}\label{IstrpSpc}
  \A_{\theta_{k+1}}=\bigl\{\;j_p^{2+k}X\in W_p/L_p^{2+k}\;\bigr|
  \;X^{(k)}_{\theta_k}\in\K_{\theta_{k+1}}\;\bigr\}\,.
\end{equation}

From this definition and \eqref{kPrlng}, we get
\begin{proposition}\label{IstrpSpcDscrb} Let $j^{2+k}_pX=(p,X^i,X^i_j,\ldots,X^i_{j_1\ldots j_{2+k}})$ in the standard     coordinates. Then $j^{2+k}_pX\in\A_{\theta_{k+1}}$ iff $(X^i,X^i_j,\ldots,X^i_{j_1\ldots j_{2+k}})$ is a solution of   
  the system of linear homogeneous algebraic equations 
  $$
    \bigl(\,D_{\sigma}(\,\psi^i_{X}\,)\,\bigr)(\theta_{k+1})=0\,,\quad i=1,2,3,4\,,\;\; 0\leq|\sigma|\leq k\,.
  $$
\end{proposition}
It follows from definition \eqref{IstrpAlgk} of the isotropy algebra $\g_{\theta_k}$ that 
$$
  \g_{\theta_k}\subset\A_{\theta_{k+1}}\;\;
  \forall\;\theta_{k+1}\in\pi_{k+1,k}^{-1}(\theta_k)\;\;
  \forall\;\theta_k\in J^k\pi\,.
$$
From the definition of $\A_{\theta_{k+1}}$, we get that
$$
  \rho_{k+2,k+1}(\A_{\theta_{k+1}})\subset \A_{\theta_k}\,.
$$

Let $f$ be a point transformation of the base of $\pi$ and let $p$ be a point of the domain of $f$. The tangent map 
$f_*:T_p\to T_{f(p)}$ generates the map 
$$
  j_p^{k+3}f:W_p/L_p^{k+2}\longrightarrow W_{f(p)}/L_{f(p)}^{k+2}\,,\quad
  j_p^{k+3}f:j_p^{k+2}X\mapsto j_{f(p)}^{k+2}\bigl(f_*(X)\bigr)\,.
$$ 
\begin{proposition}\label{TrnsfrmIstrpSpc} Let $\theta_{k+1}$ be a point of the domain of $f^{(k+1)}$. Then 
  $$
   j_p^{k+3}f(\A_{\theta_{k+1}})=\A_{f^{(k+1)}(\theta_{k+1})}\,.
  $$         
\end{proposition}
\begin{proof} Let $X$ be a vector field in the base of $\pi$ and let $\varphi_t$ be the flow of $X$. Then the condition 
$j_p^{k+2}X\in\A_{\theta_{k+1}}$ means that 
$X^{(k)}_{\theta_k}=d/dt\bigl(\varphi_t^{(k)}(\theta_k)\bigr)\bigl|_{t=0}\in\K_{\theta_{k+1}}$.
It follows that
$$
  \frac{d}{dt}(f\circ\varphi_t\circ f^{-1})^{(k)}\bigl(f^{(k)}(\theta_k)\bigr)\bigl|_{t=0}=
  \frac{d}{dt}f^{(k)}\bigl(\varphi_t^{(k)}(\theta_k)\bigr)\bigl|_{t=0}
  =f^{(k)}_*(X^{(k)}_{\theta_k})\,.
$$
It is clear that $f^{(k)}_*(\K_{\theta_{k+1}})=\K_{f^{(k+1)}(\theta_{k+1})}$ for every point $\theta_{k+1}$ of the domain of $f^{(k+1)}$. Therefore $f^{(k)}_*(X^{(k)}_{\theta_k})\in\A_{f^{(k+1)}(\theta_{k+1})} $
Thus $j_p^{k+3}f(j_p^{k+2}X)\in\A_{f^{(k+1)}(\theta_{k+1})}$. 
\end{proof}
Consider the restriction of the bilinear map $[\,\cdot\,,\,\cdot\,]:W_p/L_p^{k+2}\times W_p/L_p^{k+2}\to W_p/L_p^{k+1}$ defined by \eqref{brckt1} to $\A_{\theta_{k+1}}\times\A_{\theta_{k+1}}$.
\begin{proposition}\label{BrktIstrpSpc} 
  $$
    [\,\A_{\theta_{k+1}}\,,\,\A_{\theta_{k+1}}\,]\subset\A_{\theta_k}\,.
  $$
\end{proposition}
\begin{proof} Suppose $j^{2+k}_pX,j^{2+k}_pY\in\A_{\theta_{k+1}}$. Then
$$
  [\,j^{2+k}_pX,\,j^{2+k}_pY\,]=j^{2+k-1}_p[\,X\,,Y\,]\,.
$$ 
It is obvious that 
$$
  j^{2+k-1}_p[\,X\,,Y\,]\in\A_{\theta_k}\quad\text{iff}\quad [\,X,\,Y\,]^{(k-1)}_{\theta_{k-1}}\in\K_{\theta_k}\,,
$$ 
where $\theta_k=\pi_{k+1,k}(\theta_{k+1})$ and $\theta_{k-1}=\pi_{k,k-1}(\theta_k)$. Suppose
$$
  X = X^1\frac{\partial}{\partial x^1} + X^2\frac{\partial}{\partial x^2}\,,\quad 
  Y = Y^1\frac{\partial}{\partial x^1} + Y^2\frac{\partial}{\partial x^2}\,.
$$ 
Then 
\begin{multline*}
    [\,X,\,Y\,]^{(k-1)}_{\theta_{k-1}}
    =(\pi_{\infty,k-1})_*\bigl([\,X,\,
    Y\,]^{(\infty)}_{\theta_{\infty}}\bigr)
    =(\pi_{\infty,k-1})_*\bigl([\,X^{(\infty)},\,
    Y^{(\infty)}\,]_{\theta_{\infty}}\bigr)\\
    =(\pi_{\infty,k-1})_*\bigl([\,X^rD_r+\re_{\psi(X)},\;
    Y^1D_1+ Y^2D_2+\re_{\psi(Y)}\,]_{\theta_{\infty}}\bigr)\,,
\end{multline*}
where $\theta_{\infty}\in(\pi_{\infty, k+1})^{-1}(p)$. Taking into account the well known relations, see \cite{KLV},
$$
 [\,D_1,\,D_2\,]=[\,D_j,\,\re_{\psi}\,]=0,\;j=1,2,\quad[\,\re_{\phi},\,\re_{\psi}\,]=\re_{\{\phi,\psi\}}\,,
$$
where $\{\phi,\psi\}=\re_{\phi}(\psi)-\re_{\psi}(\phi)$, we get
   \begin{multline*}
    [\,X,\,Y\,]^{(k-1)}_{\theta_{k-1}}
    =(\pi_{\infty,k-1})_*
    \bigl(\,(X^jY^i_j-Y^jX^i_j)D_i+[\,\re_{\psi(X)},\,\re_{\psi(Y)}\,]
    \,\bigr)\bigl|_{\theta_{k-1}}\\
    =\bigl(\,(X^jY^i_j-Y^jX^i_j)D_i^{k-1}+\re^{k-1}_{\{\psi(X),\psi(Y)\}}\,\bigr)\bigl|_{\theta_{k-1}}\,.
  \end{multline*}
From \eqref{DfrmtnVlst1} and \eqref{Re}, we get
\begin{multline*}
  \{\psi(X),\psi(Y)\}^i
  =\psi^{i'}(X)\frac{\partial\psi^i(Y)}{\partial u^{i'}}
   +D_j(\psi^{i'}(X))\frac{\partial\psi^i(Y)}{\partial u^{i'}_j}\\
   -\psi^{i'}(Y)\frac{\partial\psi^i(X)}{\partial u^{i'}}
   -D_j(\psi^{i'}(Y))\frac{\partial\psi^i(X)}{\partial
   u^{i'}_j}\,.
\end{multline*}
From Proposition \ref{IstrpSpcDscrb}, we have  
\begin{gather*}
  \bigl(D_{\sigma}(\,\psi^i_{X}\,)\bigr)(\theta_{k+1})=0\quad\text{and}\quad
  \bigl(D_{\sigma}(\,\psi^i_{Y}\,)\bigr)(\theta_{k+1})=0\\
  i=1,2,3,4\,,\quad 0\leq|\sigma|\leq k\,.
\end{gather*}   
It follows that $\re^{(k-1)}_{\{\psi(X),\psi(Y)\}}\bigl|_{\theta_{k-1}}=0$. Hence, 
$$
 [\,X,\,Y\,]^{(k-1)}_{\theta_{k-1}}=\bigl((X^jY^i_j-Y^jX^i_j)D_i^{k-1}\bigr)\bigl|_{\theta_{k-1}}\,.
$$ 
This means that $[\,j^{2+k}_pX,\,j^{2+k}_pY\,]\in\A_{\theta_k}$.
\end{proof}

\subsection{Horizontal subspaces}
Let $\theta_{k+1}\in J^{k+1}\pi$, $\theta_k=\pi_{k+1,k}(\theta_{k+1})$, and $p=\pi_{k+1}(\theta_{k+1})$. A 2--dimensional subspace $H\subset\A_{\theta_{k+1}}$ is called {\it horisontal} if the natural projection
$$
  \rho_{k+2,0}\bigl|_H:H\longrightarrow T_p,\quad\rho_{k+2,0}: j_p^{k+2}X\mapsto X_p,
$$
is an isomorphism. Let $H$ be a horizontal subspace of $\A_{\theta_{k+1}}$, then 
$$
  \A_{\theta_{k+1}}=H\oplus\g_{\theta_k}\,.
$$

Every two horizontal subspaces $H,\tilde H\subset\A_{\theta_{k+1}}$ define the linear map
$$
  f_{H,\tilde H}:T_p\to\g_{\theta_k}\,,\quad
  f_{H,\tilde H}: X\mapsto (\rho_{k+2,0}|_H)^{-1}(X)
  -(\rho_{k+2,0}|_{\tilde H})^{-1}(X)\,.
$$
On the other hand, let $H\subset\A_{\theta_{k+1}}$ be a horizontal subspace and let $f:T_p\to\g_{\theta_k}$ be a linear map. Then there exists a unique horizontal subspace $\tilde H\subset\A_{\theta_{k+1}}$ such that $f=f_{H,\tilde H}$. This subspace is spanned by the $k+2$--jets $(\rho_{k+2,0}|_H)^{-1}(X)-f(X)$, $X\in T_p$.

Every horizontal subspace $H\subset\A_{\theta_{k+1}}$ generates the 2--form $\omega_H$ on $T_p$ with values in $\A_{\theta_k}$ 
\begin{equation}\label{OmgH}
  \omega_H(X_p, Y_p)=[\,(\rho_{k+2,0}\bigl|_H)^{-1}(X_p),\,(\rho_{k+2,0}\bigl|_H)^{-1}(Y_p)\,]\,,\quad
  \forall\,X_p, Y_p\in T_p.
\end{equation}

From Proposition \ref{TrnsfrmIstrpSpc} we obviously get the following
\begin{proposition} Let $f$ be a point transformation of the base of $\pi$ and let $\theta_{k+1}$ be a point of the  
  domain of the lifted transformation $f^{(k+1)}$. Then
  \begin{enumerate}
  \item \label{TrnsfrmHS} If $H$ is a horizontal subspace of $\A_{\theta_{k+1}}$, then $j_p^{k+3}f(H)$ is a   
    horizontal subspace of $\A_{f^{(k+1)}(\theta_{k+1})}$.
  \item \label{TrnsfrmOmgH}  
    $j_{f(p)}^{k+2}f\bigl(\omega_H(X_p, Y_p)\bigr)=\omega_{j_p^{k+3}f(H)}\bigl(f_*(X_p), f_*(Y_p)\bigr)
     \,,\quad\forall\,X_p, Y_p\in T_p$.
  \end{enumerate}           
\end{proposition}

%

\section{Differential invariants on $J^2\pi$}
%

\subsection{Horizontal subspaces of ${\bf\A_{\theta_2}}$}
Let  $\theta_2\in J^2\pi$, $\theta_1=\pi_{2,1}(\theta_2)$ and $p=\pi_2(\theta_2)$. 
Consider the space $\A_{\theta_2}$. It is a subspace of the space $W_p/L_p^3$. From the system of equations describing the space $\A_{\theta_2}$, see Proposition \ref{IstrpSpcDscrb}, we obtain the natural filtration of $\A_{\theta_2}$ and the corresponding graduate space $\G\A_{\theta_2}$:
$$
  \A_{\theta_2}\supset\g_{\theta_1}\supset\g_{\theta_1}^2\supset\{0\}\,,\quad
  \G\A_{\theta_2}=T_p\oplus L_p^0/L_p^1\oplus g^2\oplus\{0\},
$$
where $g^2$ is described by \eqref{g2}.

\begin{proposition} There are horizontal subspaces $H\subset\A_{\theta_2}$ satisfying the condition
  \begin{equation}\label{HS1}
    \rho_{2,1}\bigl(\,\bigl[\,j_p^3X\,,\,j_p^3Y\,\bigr]\,\bigr)=0\quad
    \forall\,j_p^3X,j_p^3Y\in H\,.
  \end{equation}
\end{proposition}
\begin{proof} First step. Let us prove that there are horizontal subspaces $H\subset\A_{\theta_2}$ satisfying the condition
 \begin{equation}\label{HS0}
  \rho_{2,0}\bigl([\,j_p^3X\,,\,j_p^3Y\,]\bigr)=0\quad\forall\,j_p^3X, j_p^3Y\in H
 \end{equation}
Let $H$ be an arbitrary horizontal subspace of $\A_{\theta_2}$. Then the formula 
$$
  t_H(X_p,Y_p)=\rho_{2,0}\bigl(\omega_H(X_p,Y_p)\bigr)\,,\quad\forall\,X_p, Y_p\in T_p\,,
$$ 
defines the tensor $t_H\in T_p\otimes(\wedge^2T_p^*)$. Let $f:T_p\to\g_{\theta_1}$ be a linear map and let $\tilde H$ be a unique horizontal subspace of $\A_{\theta_2}$ such that $f_{H,\tilde H}=f$. Then for the corresponding tensor $t_{\tilde H}$, we have 
\begin{multline*}
  t_{\tilde H}(X_p,Y_p)=\rho_{2,0}\bigl([\,j_p^3X-f(X_p)\,,\,j_p^3Y-f(Y_p)\,]\bigr)\\
  =t_H(X_p,Y_p)-\rho_{2,0}\bigl([\,j_p^3X\,,\,f(Y_p)\,]-[\,j_p^3Y\,,\,f(X_p)\,]\bigr)\\
  =t_H(X_p,Y_p)-\partial_{1,1}(f')(X_p, Y_p)\,,
\end{multline*}
where $f'=\rho_{3,1}\circ f\in L_p^0/L_p^1\otimes T_p^*$ and the operator 
$\partial_{1,1}:L_p^0/L_p^1\otimes T_p^*\to T_p\otimes\wedge^2T_p^*$ is defined by \eqref{SpncrOpr} 
$$
  \partial_{1,1}(f')(X_p, Y_p)=[\,X_p,\,f'(Y_p)\,]-[\,Y_p,\,f'(X_p)\,]\,,\quad\forall\,X_p, Y_p\in T_p\,.
$$ 
It remains to prove that the linear map $f$ can be chosen such that $t_{\tilde H}=0$. To this end consider 
the spaces $L_p^0/L_p^1$ and $L_p^1/L_p^2$. They satisfy \eqref{brckt2}. Therefore we have complex \eqref{SpncrCmplx} constructed for these spaces 
$$
  0\rightarrow L_p^1/L_p^2\xrightarrow{\partial_{2,0}} L_p^0/L_p^1\otimes T_p^*
  \xrightarrow{\partial_{1,1}} T_p\otimes\wedge^2T_p^*
  \rightarrow 0\,.
$$
It is easy to check that this complex is exact. It follows that the linear map $f$ can be chosen such that $t_{\tilde H}=0$. From $g^1_{\theta_1}=L_p^0/L_p^1$ and $L_p^1/L_p^2\neq\{0\}$, we get that there are many horizontal subspaces of $\A_{\theta_2}$ satisfying \eqref{HS0}. 

In the standard coordinates, an arbitrary horizontal subspace $H\subset\A_{\theta_2}$ has the form 
$H=\bigl\{\;j^3_0X=(\,X^i,\;h^i_{j,r}X^r,\;h^i_{j_1j_2,r}X^r,\;f^i_{j_1j_2j_3,r}X^r\,)\;\bigr\}$.
Obviously, $H$ satisfies \eqref{HS0} iff 
$$
 h^i_{j,r}=h^i_{r,j}\quad\forall\,i,j,r\,.
$$

Second step. Let $H$ be an arbitrary horizontal subspace of $\A_{\theta_2}$ satisfying \eqref{HS0}, let $f:T_p\to\g_{\theta_1}^2$ be a linear map, and let $\tilde H$ be a unique horizontal subspace of $\A_{\theta_2}$ satisfying the condition $f_{H,\tilde H}=f$. Then obviously,
\begin{equation}\label{HS0Prj}
  \rho_{3,1}(\,\tilde H\,) =\rho_{3,1}(\,H\,)\,.
\end{equation}
On the other hand, if $\tilde H$ is an arbitrary horizontal subspace of $\A_{\theta_2}$ satisfying \eqref{HS0Prj}, then $f_{H,\tilde H}\in \g_{\theta_1}^2\otimes T_p^*$. It is clear now that there are many horizontal subspaces $\tilde H\subset\A_{\theta_2}$ satisfying \eqref{HS0Prj}. All these subspaces satisfy \eqref{HS0}. Let us prove that there exists a unique subspace $\tilde H$ satisfying \eqref{HS1} among horizontal subspaces satisfying \eqref{HS0Prj}. 
Taking into account \eqref{HS0} and \eqref{IstrAlg01}, we define the tensor $t_H\in L_p^0/L_p^1\otimes(\wedge^2T_p^*)$ by the formula 
$$
  t_H(X_p,Y_p)=\rho_{2,1}\bigl(\omega_H(X_p, Y_p)\bigr)\,.
$$
Let $\tilde H$ be a horizontal subspace satisfying \eqref{HS0Prj}. Then
\begin{multline*}
 t_{\tilde H}(X_p,Y_p)=t_H(X_p,Y_p)-\rho_{2,1}\bigl([\,j_p^3X\,,\,f(Y_p)\,]-[\,j_p^3Y\,,\,f(X_p)\,]\bigr)\\
  =t_H(X_p,Y_p)-\partial_{2,1}(f')(X_p, Y_p)\,,
\end{multline*}
where $f'=\rho_{3,2}\circ f_{H,\tilde H}\in g^2\otimes T_p^*$ and the operator 
$\partial_{2,1}:g^2\otimes T_p^*\to L_p^0/L_p^1\otimes\wedge^2T_p^*$ is defined by \eqref{SpncrOpr} 
$$
  \partial_{2,1}(f')(X_p, Y_p)=[\,X_p,\,f'(Y_p)\,]-[\,Y_p,\,f'(X_p)\,]\,,\quad\forall\,X_p, Y_p\in T_p\,.
$$ 
Now, from the exactness of the following complex \eqref{SpncrCmplx}
$$
 0 = (g^2)^{(1)}\xrightarrow{\partial_{3,0}} g^2\otimes T_p^*\xrightarrow{\partial_{2,1}} L_p^0/L_p^1\otimes
 \wedge^2T_p^*\rightarrow 0\,,
$$
we obtain that there exists a unique horizontal subspace $\tilde H\subset\A_{\theta_2}$ satisfying \eqref{HS0Prj} and \eqref{HS1}.

Thus we proved that there are many horizontal subspaces  of $\A_{\theta_2}$ satisfying \eqref{HS1}.
\end{proof}
The following obvious statement is important to construct differential invariants.
\begin{proposition}\label{TrnsfrmSHS} Suppose $f$ is a point transformation of the base of $\pi$, $\theta_2$ is a point of the domain of $f^{(2)}$, and $H$ is a horizontal subspace of $\A_{\theta_2}$ satisfying \eqref{HS1}. Then the horizontal subspace $j_p^4f(H)$ of $\A_{f^{(2)}(\theta_2)}$ satisfies \eqref{HS1} too.
\end{proposition}

\subsection{The obstruction to linearization}
Let $H$ be an arbitrary horizontal subspace of $\A_{\theta_2}$ satisfying \eqref{HS1}, let $\omega_H$ be its 2--form  defined by \eqref{OmgH}, and let $p=\pi_2(\theta_2)$ Then, obviously, 
$$
  \omega_{H}\in g^2\otimes(\wedge^2T_p^*)\,.
$$
\begin{theorem}\label{ThOmgIndH} The 2--form $\omega_{H}$ is independent of the choice of a horizontal subspace $H\subset\A_{\theta_2}$ satisfying \eqref{HS1}.
\end{theorem}
\begin{proof} See section \ref{PrfThOmgIndH} of Appendix.\end{proof} 
Put  
$$
  \omega_{\theta_2}=\omega_H\,,
$$  
where $H$ is an arbitrary horizontal subspace of $\A_{\theta_2}$ satisfying \eqref{HS1}. From theorem \ref{ThOmgIndH} we get that $\omega_{\theta_2}$ is well defined. Thus for every point $\theta_2\in J^2\pi$, we define in the natural way the 2--form $\omega_{\theta_2}$ on $T_p$ with values in $g^2$. This means that the following statement holds.
\begin{theorem}\label{DffInv} The field of tensors on $J^2\pi$ 
$$
 \omega^2:\theta_2\longmapsto\omega_{\theta_2}\,.
$$
is a differential invariant of the action of $\Gamma$ on the bundle $\pi$.
\end{theorem}

We can consider $\omega^2$ as a horizontal differential 2--form on $J^2\pi$ with values in $g^2$:
$$
  \omega^2(X,Y)=\omega_{\theta_2}\bigl((\pi_2)_*(X),(\pi_2)_*(Y)\bigr)\,,
$$
where $X$ and $Y$ are tangent vectors to $J^2\pi$ at the point $\theta_2$.

In the standard coordinates, $\omega^2$ is expressed in following way, see section \ref{SctnOmg2Expr} of Appendix,
\begin{multline}\label{Obstrtn1}
  \omega^2=\Bigl(F^1\bigl(2\frac{\partial}{\partial x^1}
  \otimes (dx^1\odot dx^1)
  +\frac{\partial}{\partial x^2}\otimes (dx^1\odot dx^2)\bigr)\\
  +F^2\bigl(2\frac{\partial}{\partial x^2}\otimes (dx^2\odot dx^2)
  +\frac{\partial}{\partial x^1}\otimes
  (dx^1\odot dx^2)\bigr)\Bigr)\\
  \otimes\,(dx^1\wedge\,dx^2)\,,
\end{multline}
where $F^1$ and $F^2$ are defined by \eqref{F1F2}.

Let $\E$ be equation \eqref{eq} and let $S_{\E}$ be the corresponding section of $\pi$. By $\omega^2_{\E}$ we denote the restriction of $\omega^2$ to the image of the section $j_2S_{\E}$. From Proposition \ref{LnrstnCndtnPr}, we get the following statement.
\begin{theorem} The equation $\E$ can be reduced to the linear form by a point transformation iff
  $\omega^2_{\E}=0$.
\end{theorem}
Thus the differential invariant $\omega^2$ is a unique obstruction to linearization of equations \eqref{eq} by point transformations.

\subsection{Derived invariants}
Applying operations of tensor algebra to the tensor $\omega_{\theta_2}$ on $T_p$, we can obtain in the natural way new  tensors on $T_p$. Indeed, applying the operation of contraction
$$
  T_p\otimes(T_p^*\odot T_p^*)\otimes(\wedge^2T_p^*)\longrightarrow T_p^*\otimes(\wedge^2T_p^*)\,,\quad
  (t^i_{jk,rs})\mapsto (t^m_{mk,rs})\,,
$$
to the tensor $(2/5)\omega_{\theta_2}$, we get the tensor 
$$
  \alpha_{\theta_2}=(F^1(\theta_2)dx^1+F^2(\theta_2)dx^2)\otimes(dx^1\wedge\;dx^2)\,.
$$
Thus the tensor field 
$$
  \alpha^2:\theta_2\longmapsto\alpha_{\theta_2}
$$
on $J^2\pi$ is a differential invariant of the action of $\Gamma$ on $\pi$. 

Taking into account that $\dim T_p=2$, we obtain that the contraction 
$$
  T_p\otimes(T_p^*\wedge T_p^*)\longrightarrow T_p^*\,,\quad (t^i_{jk})\mapsto(t^m_{mk})\,,
$$
is an isomorphism. Therefore the contraction
$$
 T_p\otimes(\wedge^2T_p^*)\otimes(\wedge^2T_p^*)\longrightarrow T_p^*\otimes(\wedge^2T_p^*)\,,\quad
  (t^i_{r_1s_1,r_2s_2})\mapsto (t^m_{ms_1,r_2s_2})\,.
$$
is isomorphism also. It is easy to check that the pseudovector of weight 2 
\begin{equation}\label{beta2}
  \beta_{\theta_2}=(F^2(\theta_2)\frac{\partial}{\partial x^1}
  -F^1(\theta_2)\frac{\partial}{\partial x^2})\otimes(dx^1\wedge\;dx^2)^2
\end{equation}
is the inverse image of the tensor $(1/2)\alpha_{\theta_2}$ under this isomorphism. This means that $\beta_{\theta_2}$ is defined in the natural way. Thus the field of pseudovectors on $J^2\pi$ 
$$
  \beta^2:\theta_2\longmapsto\beta_{\theta_2}
$$  
is a differential invariant of the action of $\Gamma$ on $\pi$.

\section{Differential invariants in $J^3\pi$}\label{Inv3}
In this section, we construct differential invariants on $(\pi_{3,2})^{-1}(\Orb_2^0)$.

\subsection{Spaces $\A_{\theta_3}$}
Let  $\theta_3\in(\pi_{3,2})^{-1}(\Orb_2^0)$, $\theta_2=\pi_{3,2}(\theta_3)$ and $p=\pi_2(\theta_2)$. 
Consider the space $\A_{\theta_3}$. It is a subspace of the space $W_p/L_p^4$. From the system of equations describing the space $\A_{\theta_3}$, see Proposition \ref{IstrpSpcDscrb}, we obtain the natural filtration of $\A_{\theta_3}$ and the corresponding graduate space $\G\A_{\theta_3}$:
\begin{equation}\label{GrIstrpSpc3}
  \A_{\theta_3}\supset\g_{\theta_2}\supset\g_{\theta_2}^2\supset\{0\}\supset\{0\}\,,\quad
  \G\A_{\theta_3}=T_p\oplus g_{\theta_2}^1\oplus g^2\oplus\{0\}\oplus\{0\},
\end{equation}
where $g_{\theta_2}^1$ is described by \eqref{g1} and $g^2$ is described by \eqref{g2}. In addition, we obtain the following statement
\begin{proposition}
In the standard coordinates, components $X^i$ and $X^i_j$ of elements of $\A_{\theta_3}$ are connected by the       equations
     \begin{equation}\label{IstrSpc3}
      \begin{aligned}
       2F^1\cdot X^1_1+F^2\cdot X^2_1+F^1\cdot X^2_2
       &=-D_1F^1\cdot X^1-D_2F^1\cdot X^2\\
       F^2\cdot X^1_1+F^1\cdot X^1_2+2F^2\cdot X^2_2
       &=-D_1F^2\cdot X^1-D_2F^2\cdot X^2,
      \end{aligned}
     \end{equation}   
     where $F^1$ and $F^2$ are defined by \eqref{F1F2}
\end{proposition}
\begin{proposition} There are horizontal subspaces $H\subset\A_{\theta_3}$ satisfying the condition
\begin{equation}\label{HS3}
  \rho_{3,1}\bigl(\,\bigl[\,j_p^4X\,,\,j_p^4Y\,\bigr]\,\bigr)=0\quad\forall\,j_p^4X, j_p^4Y\in H\,.
\end{equation}
\end{proposition}
\begin{proof} First step. Show that there exist horizontal subspaces $H\subset\A_{\theta_3}$ satisfying the condition
\begin{equation}\label{HS03}
  \rho_{3,0}\bigl(\,\bigl[\,j_p^4X\,,\,j_p^4Y\,\bigr]\,\bigr)=0\quad\forall\,j_p^4X, j_p^4Y\in H
\end{equation}
To this end consider two arbitrary horizontal subspaces $H$ and $\tilde H$ of $\A_{\theta_3}$. They generate the linear map $f_{H,\tilde H}\in \g_{\theta_2}\otimes T_p^*$. Then $\rho_{4,1}\circ f_{H,\tilde H}\in g_{\theta_2}^1\otimes T_p^*$. Now the existence of horizontal subspaces satisfying condition \eqref{HS03} follows from \eqref{Dim_g1} and 
the exact sequence
$$
  0\rightarrow (g_{\theta_2}^1)^{(1)}\xrightarrow{\partial_{2,0}} g_{\theta_2}^1\otimes T_p^*
  \xrightarrow{\partial_{1,1}} T_p\otimes\wedge^2T_p^*
  \rightarrow 0. 
$$

Second step. Let $H$ be an arbitrary horizontal subspace of $\A_{\theta_3}$ satisfying \eqref{HS03}, let $f:T_p\to\g_{\theta_2}^2$ be a linear map, and let $\tilde H$ be a unique horizontal subspace of $\A_{\theta_3}$ satisfying the condition $f_{H,\tilde H}=f$. Then obviously,
\begin{equation}\label{HS03Prj}
  \rho_{4,1}(\,\tilde H\,) =\rho_{4,1}(\,H\,)\,.
\end{equation}
On the other hand, if $\tilde H$ is an arbitrary horizontal subspace of $\A_{\theta_3}$ satisfying \eqref{HS03Prj}, then $f_{H,\tilde H}\in \g_{\theta_2}^2\otimes T_p^*$. It is clear now that there are many horizontal subspaces $\tilde H\subset\A_{\theta_3}$ satisfying \eqref{HS03Prj}. From \eqref{GrIstrpSpc3}, we get that $\rho_{4,2}\circ f_{H,\tilde H}\in g^2\otimes T_p^*$. Finally, from the exact sequence
$$
  0 = (g^2)^{(1)}\xrightarrow{\partial_{3,0}} g^2\otimes T_p^*
  \xrightarrow{\partial_{2,1}} (L_p^0/L_p^1)\otimes\wedge^2T_p^*
  \rightarrow 0\,,
$$
we obtain that there exists a unique horisontal subspace satisfying \eqref{HS3} among subspaces $\tilde H$ satisfying \eqref{HS03Prj}.
\end{proof}

Let $H$ be a horizontal subspace of $\A_{\theta_3}$ satisfying \eqref{HS3} and let $j_p^4X, j_p^4Y\in H$. Then
$[j_p^4X, j_p^4Y]\in \g_{\theta_1}^2$. Obviously, $\rho_{3,2}([j_p^4X, j_p^4Y])$ is element of $g^2$ and it is equal to $\omega^2(X_{\theta_2}^{(2)},Y_{\theta_2}^{(2)})$. It follows from \eqref{FstPrlng1} that this element of $g^2$ defines $[j_p^4X, j_p^4Y]$ uniquely. Thus we get the following 
\begin{remark} The bracket between vectors of a horizontal subspace of $\A_{\theta_3}$ satisfying \eqref{HS3} does not lead to a new differential invariant differing of $\omega^2$. 
\end{remark}

\subsection{Invariant form $\omega^3$}
Taking into account the previous remark, we will investigate the bracket between vectors of a horizontal subspace $H\subset\A_{\theta_3}$ and elements of the algebra $\g_{\theta_2}$ to construct new differential invariants. 

To minimize the arbitrariness in our constructions, we will consider a horizontal subspace $H$ satisfying \eqref{HS3} and the element $[j_p^4X, j_p^4Y]\in\g_{\theta_1}^2\subset\g_{\theta_1}$, where $j_p^4X, j_p^4Y\in H$. 

Let $j_p^4U, j_p^4Z\in H$, then $w=\bigl[ j_p^3Z,[j_p^4X, j_p^4Y]\bigr]\in\g_{\theta_0}$  and $[j_p^2U,w]\in W_p/L_p^1$. There exists a unique vector $j_p^4\tilde Z\in H$ such that $\tilde Z_p=\rho_{1,0}([j_p^2U,w])$. Then $[j_p^2U,w]-j_p^1\tilde Z$ is element of $L_p^0/L_p^1$. Thus the formula  
$$
  t_H(X_p,Y_p,Z_p,U_p)=[j_p^2U,w]-j_p^1\tilde Z\quad\forall\,U_p,Z_p,X_p,Y_p\in T_p
$$
defines the tensor
$$
  t_H\in (T_p\otimes T_p^*)\otimes T_p^*\otimes T_p^*\otimes(T_p^*\wedge T_p^*)\,. 
$$ 
This tensor depends on the choice of a horizontal subspace $H$ satisfying \eqref{HS3}. We transform this tensor to obtain a new tensor independent of this choice. To this end consider $t_H$ in detail.
\begin{lemma}\label{Lmm} 
$$
   t_{H}\in T_p\otimes( T_p^*\odot T_p^*\odot T_p^*)\otimes(T_p^*\wedge T_p^*)  
$$
\end{lemma}
\begin{proof} See section \ref{PrfLmm} of Appendix \end{proof}

Recall that the ideal $g^2=\g_{\theta_0}\cap (L_p^1/L_p^2)$ of the isotropy algebra $\g_{\theta_0}$ is defined by \eqref{g2}. It can be considered as a subspace of $T_p\otimes( T^*_p\odot T^*_p)$. There exists a natural projection
$$
  \mu:T_p\otimes( T_p^*\odot T_p^*)\longrightarrow g^2\,,\quad
  \mu:(X^i_{jk})\mapsto \bigl(\,\frac13\,(\delta^i_j\,X^r_{kr}+\delta^i_k\,X^r_{jr})\,\bigr)\,,
$$
where $\delta^i_j$ is the Kronecker symbol. This projection generates the natural projection
$$
  \tilde\mu:T_p\otimes( T_p^*\odot T_p^*)\otimes T_p^*\otimes(T_p^*\wedge T_p^*)\longrightarrow
  g^2\otimes T_p^*\otimes(T_p^*\wedge T_p^*).
$$
Taking into account that
$$
  T_p\otimes( T_p^*\odot T_p^*\odot T_p^*)\otimes(T_p^*\wedge T_p^*)\subset
  T_p\otimes( T_p^*\odot T_p^*)\otimes T_p^*\otimes(T_p^*\wedge T_p^*),
$$
we can consider the tensor $\tilde\mu(t_H)\in g^2\otimes T_p^*\otimes(T_p^*\wedge T_p^*)$ as a 1-form with values in $g^2\otimes(T_p^*\wedge T_p^*)$. Then the contraction 
\begin{gather*}
  \bigl(T_p\otimes (T_p^*\wedge T_p^*)^2\bigr)\kch\Bigl(\bigl(g^2\otimes(T_p^*\wedge T_p^*)\bigr)
  \otimes T_p^*\Bigr)\subset g^2\otimes(T_p^*\wedge T_p^*)^3\,,\\
  (t^i_{r_1s_1r_2s_2})\kch(p^i_{jk,r_3s_3,l})=(t^m_{r_1s_1r_2s_2}p^i_{jk,r_3s_3,m})\,,
\end{gather*}
defines the new tensor
$$
  \omega^3_H=\beta_{\theta_2}\kch\tilde\mu(t_H)\in g^2\otimes(T_p^*\wedge T_p^*)^3,
$$  
where $\beta_{\theta_2}\in T_p\otimes (T_p^*\wedge T_p^*)^2$ and is defined by \eqref{beta2}. 
\begin{theorem}\label{ThrmOmg3H}
  The tensor $\omega_H^3$ is independent of the choice of a horizontal subspace $H\subset\A_{\theta_3}$ satisfying
  \eqref{HS3}.
\end{theorem}
\begin{proof} See section \ref{SctnThrmOmg3H} of Appendix \end{proof} 

By $\omega_{\theta_3}$ we denote the tensor $3\omega_H^3$. From the proof of Theorem \ref{ThrmOmg3H}, we have that $\omega_{\theta_3}$ is described by the following formula in the standard coordinates\,:
$$
  \omega_{\theta_3}=\bigl(\Psi^1(\theta_3)\cdot e_1+\Psi^2(\theta_3)\cdot e_2\bigr)\otimes(dx^1\wedge dx^2)^3,
$$
where $e_1$ and $e_2$ are generators of $g^2$ defined by \eqref{Gnrtrs_g2}, 
\begin{align*}
  \Psi^1(\theta_3)&=3(\omega^3_H)^2_{12}/(\lambda^3)= - (F^1)^2u^2 + 2F^1F^2u^1 - 3(F^2)^2u^0\\
        &\phantom{=3(\omega^3_H)^2_{12}/(\lambda^3)= - (F^1)^2u^2 }- F^1F^1_y + 4F^1F^2_x - 3F^1_xF^2,\\
  \Psi^2(\theta_3)&=3(\omega^3_H)^1_{12}/(\lambda^3)= - 3(F^1)^2u^3 + 2F^1F^2u^2 - (F^2)^2u^1\\ 
        &\phantom{=3(\omega^3_H)^1_{12}/(\lambda^3)= - 3(F^1)^2u^3 }+ 3F^1F^2_y - 4F^1_yF^2 + F^2F^2_x,
\end{align*}

and $\theta_3=(\,x^j,\,u^i,\,\ldots,\,u^i_{j_1j_2j_3})$. It is clear that the tensor $\omega_{\theta_3}$ is defined by the point $\theta_3\in(\pi_{3,2})^{-1}(\Orb_2^0)\subset J^3\pi$ in the natural way. Therefore, the map
\begin{equation}\label{omg3}
  \omega^3:\theta_3\longmapsto\omega_{\theta_3}\quad\forall\,\theta_3\in(\pi_{3,2})^{-1}(\Orb_2^0)  
\end{equation}
is a differential invariant.

\subsection{Derived invariants}
Applying the contraction
\begin{gather*}
  T_p\otimes(T_p^*\odot T_p^*)\otimes(T_p^*\wedge T_p^*)^3\longrightarrow
  T_p^*\otimes(T_p^*\wedge T_p^*)^3\,,\\
  (t^i_{jk,r_1s_1r_2s_2r_3s_3})\mapsto (t^m_{mk,r_1s_1r_2s_2r_3s_3})\,,
\end{gather*}
to $\omega_{\theta_3}$, we obtain in the natural way the new tensor at the point $\theta_3\in(\pi_{3,2})^{-1}(\Orb_2^0)$
\begin{equation}\label{alph3}
  \alpha_{\theta_3}=(\Psi^1dx^1+\Psi^2dx^2)\otimes(dx^1\wedge dx^2)^3.
\end{equation}
Therefore the field of tensors on $(\pi_{3,2})^{-1}(\Orb_2^0)$
$$
  \alpha^3: {\theta_3}\longmapsto  \alpha_{\theta_3}
$$
is a differential invariant of the action of $\Gamma$ on $\pi$.

The contraction of $\beta_{\theta_2}$ and $\alpha_{\theta_3}$, where $\theta_2=\pi_{3,2}(\theta_3)$,  gives in the natural way the next tensor
\begin{equation}\label{nu}
 \nu_{\theta_3}=\frac13(\beta_{\theta_2}\kch\alpha_{\theta_3})=F^3(dx^1\wedge dx^2)^5,
\end{equation}
where $F^3$ is defined by \eqref{F3}.
Therefore the tensor field on $(\pi_{3,2})^{-1}(\Orb_2^0)$ 
$$
  \nu:\theta_3\longmapsto \nu_{\theta_3}
$$
is a differential invariant of the action of $\Gamma$ on $\pi$. First this invariant was obtained by R. Liouville in \cite{Lvll}.

The contraction $T_p\otimes(T_p^*\wedge T_p^*)\to T_p^*$ is an isomorphism. Therefore the contraction
$$
  T_p\otimes(T_p^*\wedge T_p^*)\otimes(T_p^*\wedge T_p^*)^3\longrightarrow T_p^*\otimes(T_p^*\wedge T_p^*)^3
$$
is an isomorphism too. The tensor
$$
  \beta_{\theta_3}=(\Psi^2\frac{\partial}{\partial x^1}
  -\Psi^1\frac{\partial}{\partial x^1})\otimes(dx^1\wedge dx^2)^4
$$
is the inverse image of the tensor $\alpha_{\theta_3}\in T_p^*\otimes(T_p^*\wedge T_p^*)^3$. Therefore the tensor field on $(\pi_{3,2})^{-1}(\Orb_2^0)$
$$
  \beta^3:{\theta_3}\longmapsto \beta_{\theta_3}
$$
is a differential invariant of the action of $\Gamma$ on $\pi$.

For every point $\theta_3\in\Orb_3^0$, the tensors $\nu_{\theta_3}$, $\beta_{\theta_2}$, and $\beta_{\theta_3}$, where $\theta_2=\pi_{3,2}(\theta_3)$, generate in the natural way the vectors $\xi_{1_{\theta_3}}$ and $\xi_{2_{\theta_3}}$:
\begin{equation}\label{XiZeta}
  \xi_{1_{\theta_3}}=\frac{1}{(F^3)^{2/5}}(F^2\frac{\partial}{\partial x^1}-F^1\frac{\partial}{\partial x^2}),\quad
  \xi_{2_{\theta_3}}=\frac{1}{(F^3)^{4/5}}(\Psi^2\frac{\partial}{\partial x^1}-\Psi^1\frac{\partial}{\partial x^2})
\end{equation}
Therefore the fields
$$
  \xi_1^3:\theta_3\longmapsto\xi_{1_{\theta_3}}\,,\quad\xi_2^3:\theta_3\longmapsto\xi_{2_{\theta_3}}\,,\quad
  \forall\,\theta_3\in\Orb_3^0
$$
are differential invariants of the action of $\Gamma$ on $\pi$.
\begin{proposition} For every point $\theta_3\in\Orb_3^0$, the vectors $\xi_{1_{\theta_3}},\,\xi_{2_{\theta_3}}$ of $T_p$ are linearly independent.
\end{proposition}
\begin{proof} It is easy to calculate that $-F^2\Psi^1+F^1\Psi^2=-3F^3\neq 0$ for every $\theta_3\in\Orb_3^0$.
\end{proof} 

Now we can define the vector fields $\xi_1$ and $\xi_2$ on $(\pi_{\infty,3})^{-1}(\Orb_3^0)$ by the formulas 
\begin{equation}\label{XiZeta1}
  \xi_1=\frac{1}{(F^3)^{2/5}}(F^2D_1-F^1D_2),\quad
  \xi_2=\frac{1}{(F^3)^{4/5}}(\Psi^2D_1-\Psi^1D_2),
\end{equation}
where $D_j$ is the operator of total derivative w.r.t. $x^j$, $j=1,2$, see \eqref{Dj}. It is clear that these vector fields are invariant w.r.t. every lifted point transformation $f^{(\infty)}$. This means that $\xi_1$ and $\xi_2$ are differential invariants of the action of $\Gamma$ on $\pi$.

It follows from the last proposition that the vector fields $\xi_1$ and $\xi_2$ are linear independent in every point $\theta_{\infty}\in (\pi_{\infty,3})^{-1}(\Orb_3^0)$. 

\section{Scalar differential invariants}
\subsection{Algebra of scalar differential invariants}
Recall that a function defined in $J^k\pi$ and invariant w.r.t. all lifted point transformations $f^{(k)}$ is {\it a scalar differential invariant of order $k$}. 

In this section, we construct scalar differential invariants in the bundles $(\pi_{k,3})^{-1}(\Orb_3^0)\subset J^k\pi$, $k>3$.

By $A_k$ we denote algebra of all scalar differential invariants of order $k$ in $(\pi_{k,3})^{-1}(\Orb_3^0)$, $k > 3$.
It is clear that if $I\in A_k$, then $(\pi_{k+1,k})_*(I)\in A_{k+1}$. We will identify these invariants. Thus we have the filtration
$$
   A_0\subset  A_1\subset\ldots\subset A_k\subset\ldots
$$

Clearly, that every function of $k$-order scalar differential invariants is a $k$-order scalar differential invariant.
Let $\{\,I^1,\,\ldots,\,I^{N_k}\,\}$ be a maximal collection of $k$-order functionally independent scalar differential invariants. Then this collection generates $A_k$, that is every invariant $I\in A_k$ is some function of $I^1,\,\ldots,\,I^{N_k}$.

Let $\theta_k$ be a generic point of $J^k\pi$ and let $p=\pi_k(\theta_k)$. Then obviously the following formula holds.
\begin{equation}\label{NmbrScInv}
  \dim J^k\pi=\dim(W_p/L_p^{k+2})-\dim\g_{\theta_k}+ N_k\,.
\end{equation}
In section \ref{SctnIstrAlg}, we obtained the following results for a generic point $\theta_k$: $\dim\g_{\theta_k}=6$ if $k=0,1$,  $\dim\g_{\theta_k}=4$ if $k=2$, and  $\dim\g_{\theta_k}=0$ if $k\geq 3$. Using formula \eqref{NmbrScInv}, we get now the following table
$$
  \begin{array}{c|c|c|c|c}
  k & \dim J^k\pi & \dim(W_p/L_p^{k+2}) & \dim\g_{\theta_k} & N_k\\
  \hline
  0 & 8   & 14 & 6 & 0\\
  1 & 16  & 22 & 6 & 0\\
  2 & 28  & 32 & 4 & 0\\
  3 & 44  & 44 & 0 & 0\\
  4 & 64  & 58 & 0 & 6\\
  5 & 88  & 74 & 0 & 14\\
  \ldots & \ldots & \ldots & \ldots & \ldots\\ 
  k & 2k^2+6k+8 & k^2+7k+14 & 0 & k^2-k-6
  \end{array}
$$
From this table, we get 
\begin{proposition}
 \begin{enumerate}
  \item The algebra $A_k$, $0\leq k\leq 3$, is trivial, that is it consists of constants.
  \item The algebra $A_k$, $k\geq 4$, is generated by $k^2-k-6$ functionally independent scalar differential invariants     of order $k$. In particular, $A_4$ is generated by 6 independent invariants and $A_5$ is generated by 14 independent      invariants.
 \end{enumerate}
\end{proposition}

\subsection{Generators}
Let $\theta_4$ be a point of $(\pi_{4,3})^{-1}(\Orb_3^0)\subset J^4\pi$, $\theta_3=\pi_{4,3}(\theta_4)$ and $p=\pi_4(\theta_4)$. Consider the space $\A_{\theta_4}$. From Proposition \ref{AlgIstr3}, we have that $\g_{\theta_3}=\{0\}$. On the other hand, from system defining $\A_{\theta_4}$, see Proposition \ref{IstrpSpcDscrb}, we have that $\A_{\theta_4}$ contains horizontal subspaces. Thus, $\A_{\theta_4}$ is a horizontal subspace. By $\omega_{\theta_4}$ we denote the 2--form $\omega_{\A_{\theta_4}}$ on $T_p$ with values in $\A_{\theta_3}$ defined by formula \eqref{OmgH}. Then $\rho_{4,0}\circ\omega_{\theta_4}$ is a 2--form on $T_p$ with values in $T_p$.
Decomposing the vector $\rho_{4,0}\circ\omega_{\theta_4}(\,\xi_{1_{\theta_3}},\,\xi_{2_{\theta_3}}\,)$ over the base $\{\,\xi_{1_{\theta_3}},\,\xi_{2_{\theta_3}}\,\}$ of $T_p$
$$
  \rho_{4,0}\circ\omega_{\theta_4}(\,\xi_{1_{\theta_3}},\,\xi_{2_{\theta_3}}\,)
  =I^1(\theta_4)\xi_{1_{\theta_3}}+I^2(\theta_4)\xi_{2_{\theta_3}},
$$
we obtain the numbers $I^1(\theta_4)$ and $I^2(\theta_4)$ in the natural way. Thus the functions  
$$
  I^1:\theta_4\mapsto I^1(\theta_4),\quad I^2:\theta_4\mapsto I^2(\theta_4)
$$ 
are scalar differential invariants on $(\pi_{4,3})^{-1}(\Orb_3^0)\subset J^4\pi$.

Next scalar invariants can be obtained in the following way.
Let $j_p^4Z$ be the vector of the horizontal subspace $\rho_{5,4}(\A_{\theta_4})$ of $\A_{\theta_3}$ such that
$Z_p=\rho_{4,0}\bigl(\omega_{\theta_4}(\xi_{1_{\theta_3}},\,\xi_{2_{\theta_3}})\bigr)$.
Then $\omega_{\theta_4}(\xi_{1_{\theta_3}},\,\xi_{2_{\theta_3}})-j_p^4Z\,\in\g_{\theta_2}$.
It follows that $\rho_{4,1}\bigl(\omega_{\theta_4}(\xi_{1_{\theta_3}},\,\xi_{2_{\theta_3}})-j_p^4Z\,\bigr)\,\in T_p\otimes T_p^*$. By $\Delta$ we denote this element of $T_p\otimes T_p^*$.
We have that $\Delta(\xi_{1_{\theta_3}})$ and $\Delta(\xi_{2_{\theta_3}})$ are vectors of $T_p$. The decompositions of these vectors over the base $\{\,\xi_{1_{\theta_3}},\,\xi_{2_{\theta_3}}\,\}$ 
$$
  \Delta(\xi_{1_{\theta_3}})=I^3(\theta_4)\xi_{1_{\theta_3}}+I^4(\theta_4)\xi_{2_{\theta_3}},\quad
  \Delta(\xi_{2_{\theta_3}})=I^5(\theta_4)\xi_{1_{\theta_3}}+I^6(\theta_4)\xi_{2_{\theta_3}}
$$
give the numbers $I^j(\theta_4)$, $j=3,4,5,6$. Clearly that these numbers are constructed in the natural way. Thus the functions 
$$
  I^j:\theta_4\mapsto I^j(\theta_4),\quad j=3,4,5,6,
$$
are new scalar differential invariants on $(\pi_{4,3})^{-1}(\Orb_3^0)\subset J^4\pi$.
The following theorem can be proved by direct calculations with the help of computer algebra.
\begin{theorem} The collection $\{\,I^1,\,I^2,\,\ldots,\,I^6\,\}$ is a maximal collection of functionally independent  
  invariants of the algebra $A_4$. 
\end{theorem}
It is clear that if $I$ is a $k$-order scalar differential invariant, then its Lie derivative $\xi_j(I)$ along the invariant vector field $\xi_j$, $j=1,2$, is a $k+1$-order scalar differential invariant. 
The following theorem can be proved also by direct calculations with the help of computer algebra.
\begin{theorem} The algebra $A_5$ is generated by the invariants $I^k$, $\xi_j(I^k)$, $j=1,2$, $k=1,2,\ldots,6$.
  In particular, the collection of 14 invariants: $I^k,\,\xi_1(I^k)$, $k=1,2,\ldots,6$, $\xi_2(I^5)$, and $\xi_2(I^6)$ 
  is  a maximal collection of functionally independent invariants of the algebra $A_5$. 
\end{theorem}

\section{The equivalence problem}
\subsection{}
Suppose $\E_1$ and $\E_2$ are equations of form \eqref{eq}, $a^i_1$ and $a^i_2$, $i=0,1,2,3$, are the coefficients of these equations respectively. Consider equations \eqref{CffTrnsfrm} describing transformation of coefficients of equations \eqref{eq} under point transformations. From these equations, we obtain the system of 2--order PDEs for a point transformation $f$
\begin{gather*}
 F^m(\,f^i,f^i_j,f^i_{jk}\,)=a^m_1-\Phi^m\bigl(\,a^0_2(f^1,f^2),\ldots,a^3_2(f^1,f^2),\, f^i_j,f^i_{jk}\,\bigr)=0\,,\\
 f^1_1f^2_2-f^2_1f^1_2\neq 0\qquad m=0,1,2,3\,.
\end{gather*}
We denote this system by $\Y(\E_1,\E_2)$. The equations $\E_1$ and $\E_2$ are locally equivalent iff $\Y(\E_1,\E_2)$ has a solution.

Let $\tau:\R^2\times\R^2\longrightarrow\R^2$ be a product bundle, $\tau_k:J^k\tau\longrightarrow\R^2$ the bundle of $k$--jets of sections of $\tau$, and $\tau_{k_1,k_2}:J^{k_1}\tau\longrightarrow J^{k_2}\tau$, $k_1 > k_2$, the natural projection sending a $k_1$--jet to its $k_2$--jet. We considered the system $\Y(\E_1,\E_2)$ as a submanifold of $J^2\tau$ and we consider a solution $f$ of $\Y(\E_1,\E_2)$ as a section $f$ of $\tau$ such that the image of the section $j_2f$ of $\tau_2$ belongs to $\Y(\E_1,\E_2)$.

Consider an arbitrary $k$--order PDE system $\Y\subset J^k\tau$. Let $y\in\Y$ and $y'=\tau_{k,k-1}(y)$. The tangent space to $\Y\cap\tau^{-1}_{k,k-1}(y')$ at the point $y$ is called the {\it symbol of $\Y$ at the point $y$} and is denoted by $\Smbl_{y}\Y$. It is easy to prove that for every point $y\in\Y(\E_1,\E_2)$, the symbol $\Smbl_y\Y(\E_1,\E_2)$ coincides with the vector space $g^2$ describing by \eqref{g2}.

The $r$-th prolongation, $r=1,2,\ldots$, of $\Y(\E_1,\E_2)$ is defined as the submanifold $\Y(\E_1,\E_2)^{(r)}\subset J^{2+r}\tau$ describing by the  system of equations
\begin{align*}
  &D_{\sigma}F^m=0,\quad m=0,1,2,3,\;0\le |\sigma|\le r,\\
  &f^1_1f^2_2-f^2_1f^1_2\neq 0\,,
\end{align*}
where $D_{\sigma}=D_{(j_1,j_2,\ldots,j_{|\sigma|})}=D_{j_1}\circ\cdots\circ D_{j_{|\sigma|}}$ and 
$D_j$ is the operator of total derivative w.r.t. $x^j$ in the bundle $J^{\infty}\tau$. 
Let $j_p^{2+r}f\in \Y(\E_1,\E_2)^{(r)}$. Then $f^{(r)}(j_p^rS_{\E_1})=j_{f(p)}^rS_{\E_2}$. Taking into account that the $r$--jet $f^{(r)}(j_p^rS_{\E_1})$ is defined by the $2+r$--jet $j_p^{2+r}f$, we will say that $2+r$--jet $j_p^{2+r}f$ transforms the $r$--jet $j_p^rS_{\E_1}$ to the $r$--jet $j_{f(p)}^rS_{\E_2}$. 

It is easy to prove that for every point $y\in\Y(\E_1,\E_2)^{(r)}$, the symbol $\Smbl_y\Y(\E_1,\E_2)^{(r)}$ is equal to $\{\,0\,\}$ if $r\geq 1$.

The following theorem holds, see \cite{Krnsh}.
\begin{theorem}\label{FntPDE} Let $\Y\subset J^k\tau$ be a PDE system. Assume that
  \begin{enumerate}
   \item $\Smbl_{y}\Y=\{\, 0\,\}$ for every $y\in\Y$,
   \item $\tau_{k+1,k}\bigl|_{\Y^{(1)}}:\Y^{(1)}\to\Y$ is surjective.
  \end{enumerate} 
  Then for every $y\in\Y$ there is a solution $f$ of $\Y$ such that $j^k_pf=y,\; p=\tau_k(y)$.
\end{theorem}

\subsection{}
Let $\E$ be an equation of form \eqref{eq} and $S_{\E}$ be the section of $\pi$ identified with this equation. We can consider the restrictions of a scalar differential invariant $I$ of order $k$ to the image of the section $j_kS_{\E}$ as a function of $x^1$ and $x^2$ in the domain of $S_{\E}$. This function is called a scalar differential invariant of order $k$ of the equation $\E$ and is denoted by $I_{\E}$. By $A^k_{\E}$ we denote the algebra of all scalar differential invariants of order $k$ of the equation $\E$. 

Let $p$ be a point of the domain of $S_{\E}$.  We say that $p$ is {\it regular} if there exists a neighborhood $U_p$ of $p$ such that the image of the restriction $j_3S_{\E}|_{U_p}$ belongs to $\Orb_3^0$. We will say that the neighborhood $U_p$ is {\it regular} too.
We will solve the equivalence problem in neighborhoods of regular points.

Let $p$ be a regular point of $\E$. Then it is possible three cases:
\begin{enumerate}
\item In some neighborhood of $p$, invariants $I^n_{\E}$, $n=1,2,\ldots,6$, are constants.
\item Among the invariants $I^1_{\E},\ldots,I^6_{\E}$, there is a nontrivial invariant generating $A^5_{\E}$ in some 
  neighborhood of $p$.
\item Among the invariants $I^n_{\E}$, $(\xi_j(I^n))_{\E}$, $n=1,2,\ldots,6$, $j=1,2$, there are two   
  functionally independent invariants in some neighborhood of $p$.
\end{enumerate}

In the first case, the equivalence problem is solved by
\begin{theorem} Suppose $\E_1$ and $\E_2$ are equations of form \eqref{eq}, $p_1$ and $p_2$ are their regular  
  points, and the invariants $I^k_{\E_1}$ and $I^k_{\E_2}$, $k=1,2,\ldots,6$, are constants in some neighborhoods of 
  $p_1$ and $p_2$ respectively. Then there exists a point transformation of neighborhoods of the points $p_1$ and $p_2$    transforming $\E_1$ to $\E_2$ and taking $p_1$ to $p_2$ iff 
  $$
    I^k_{\E_1}(p_1)=I^k_{\E_2}(p_2)\quad\forall\,k.
  $$
\end{theorem}
\begin{proof} The necessity is obvious. Prove the sufficiency. Consider regular neighborhoods $U_{p_1}$ of $p_1$ and $U_{p_2}$ of $p_2$ so that $I^k_{\E_1}|_{U_{p_1}}$ and $I^k_{\E_2}|_{U_{p_2}}$ are constants, $k=1,2,\ldots,6$. Restrict $\E_1$ to $U_{p_1}$ and $\E_2$ to $U_{p_2}$. By $\Y$ we denote the PDE $\Y(\E_1|_{U_{p_1}},\E_2|_{U_{p_2}})$. 

Consider the equation $\Y^{(3)}$. We have that $\Smbl_y\Y^{(3)}=\{0\}$ for every point $y\in\Y^{(3)}$. 

Check that the projection $\tau_{6,5}:\bigl(\Y^{(3)}\bigr)^{(1)}=\Y^{(4)}\longrightarrow\Y^{(3)}$ is a surjection. Suppose $j_p^5f\in\Y^{(3)}$. Then $j_p^5f$ takes $j_p^3S_{\E_1}$ to $j_{f(p)}^3S_{\E_2}$. Taking into account that the isotropy algebra $\g_{\theta_3}=\{0\}$ for every point $\theta_3\in\Orb_3^0$, we get that there exists a unique $5$-jet of point transformations taking $j_p^3S_{\E_1}$ to $j_{f(p)}^3S_{\E_2}$. It follows from the condition $I^k_{\E_1}|_{U_{p_1}}=I^k_{\E_2}|_{U_{p_2}}\;\forall\,k$ that the jets $j_p^4S_{\E_1}$ and $j_{f(p)}^4S_{\E_2}$ belong to the same orbit of $J^4\pi$. Hence there exists a $6$-jet $j_p^6f'\in\Y^{(4)}$ transforming $j_p^4S_{\E_1}$ to $j_{f'(p)}^4S_{\E_2}$. Obviously, $\tau_{6,5}(j_p^6f')=j_p^5f$. Thus the projection $\bigl(\Y^{(3)}\bigr)^{(1)}\longrightarrow\Y^{(3)}$ is a surjection.

The $3$-jets $j_{p_1}^3S_{\E_1}$ and $j_{p_2}^3S_{\E_2}$ belong to the orbit $\Orb_3^0$. Hence there exists a (unique) jet $j_{p_1}^5f$ of point transformations taking $j_{p_1}^3S_{\E_1}$ to $j_{p_2}^3S_{\E_2}$. Now it follows from Theorem \ref{FntPDE} that there exists a solution $f'$ of the equation $\Y$ such that $j_{p_1}^5f'=j_{p_1}^5f$.
\end{proof}

In the second case, the equivalence problem is solved by
\begin{theorem} Suppose $\E_1$ and $\E_2$ are equations of form \eqref{eq}, $p_1$ and $p_2$ are their regular points. Suppose $J_{\E_1}\in \{\,I^1_{\E_1},\ldots,I^6_{\E_1}\,\}$, $dJ_{\E_1}\bigr|_p\neq 0$, and $J_{\E_1}$ generates $A^5_{\E_1}$ in some neighborhood of $p_1$. Then there exists a point transformation of neighborhoods of $p_1$ and $p_2$ transforming $\E_1$ to $\E_2$ and taking $p_1$ to $p_2$ iff the following
conditions hold:
  \begin{enumerate}
   \item $dJ_{\E_2}\bigr|_{p_2}\neq 0$, $J_{\E_2}$ generates $A^5_{\E_2}$, and 
     $J_{\E_1}(p)=J_{\E_2}(p_2)$.
   \item If $I^k_{\E_1}=F^k(J_{\E_1})$ and $\bigl(\xi_j(I^k)\bigr)_{\E_1}=F^k_j(J_{\E_1})$ in some neighborhood of $p_1$,
     then $I^k_{\E_2}=F^k(J_{\E_2})$ and $\bigl(\xi_j(I^k)\bigr)_{\E_2}=F^k_j(J_{\E_2})$ in some  
     neighborhood of $p_2$, $k=1,2,\ldots,6$, $j=1,2$.
  \end{enumerate}
\end{theorem}
\begin{proof} The necessity is obvious. Prove the sufficiency. It is clear that there exists a neighborhood $V$ of the point $J_{\E_1}(p_1)=J_{\E_2}(p_2)$ in $\R$ such that $J_{\E_1}$ generates $A^5_{\E_1}$ in $U_{p_1}=J_{\E_1}^{-1}(V)$ and conditions (1) and (2) are satisfied for $\E_2$ in $U_{p_2}=J_{\E_2}^{-1}(V)$. Let $\Y=\Y(\E_1|_{U_{p_1}},\E_2|_{U_{p_2}})$. Then $\Smbl_y\Y^{(4)}=\{0\}$ for every point $y\in\Y^{(4)}$.

Check that the projection $\tau_{7,6}:\bigl(\Y^{(4)}\bigr)^{(1)}=\Y^{(5)}\longrightarrow\Y^{(4)}$ is a surjection. Let $j_{p}^6f\in\Y^{(4)}$. This means that $j_{p}^6f$ transforms $j_{p}^4S_{\E_1}$ to $j_{f(p)}^4S_{\E_2}$.
Taking into account that the isotropy algebra $\g_{\theta_4}=\{0\}$ for every point $\theta_4\in(\pi_{4,3})^{-1}(\Orb_3^0)$, we get that $j_{p}^6f$ is a unique $6$-jet transforming $j_{p}^4S_{\E_1}$ to $j_{f(p)}^4S_{\E_2}$. The jets $j_{p}^4S_{\E_1}$ and $j_{f(p)}^4S_{\E_2}$ belongs to the same orbit. Hence $J_{\E_1}(j_{p}^4S_{\E_1})=J_{\E_2}(j_{f(p)}^4S_{\E_2})$. From condition (2) of the theorem, we get that $I_{\E_1}(j_{p}^5S_{\E_1})=I_{\E_2}(j_{f(p)}^5S_{\E_2})$ for every $I\in A_5$. This means that
the $5$-jets $j_{p}^5S_{\E_1}$ and $j_{f(p)}^5S_{\E_2}$ belong to the same orbit. Hence there exists a $7$-jet $j_{p}^7f'\in\Y^{(5)}$ transforming $j_{p}^5S_{\E_1}$ to $j_{f'(p)}^5S_{\E_2}$. Obviously $\tau_{7,6}(j_p^7f')=j_p^6f$.
Thus the projection $\bigl(\Y^{(4)}\bigr)^{(1)}\longrightarrow\Y^{(4)}$ is a surjection.

It follows from conditions of the theorem that the $4$-jets $j_{p_1}^4S_{\E_1}$ and $j_{p_2}^4S_{\E_2}$ belong to the same orbit. Hence there exists a (unique) jet $j_{p_1}^6f$ transforming $j_{p_1}^4S_{\E_1}$ to $j_{p_2}^4S_{\E_2}$. Now it follows from Theorem \ref{FntPDE} that there exists a solution $f'$ of the equation $\Y$ such that $j_p^6f'=j_p^6f$.
\end{proof}

In the last case, the equivalence problem is solved by the following theorem, which is proved in the same way as the previous one. 
\begin{theorem} Suppose $\E_1$ and $\E_2$ are equations of form \eqref{eq}, $p_1$ and $p_2$ are their regular points. 
  Suppose invariants $J^1_{\E_1}$, $J^2_{\E_1}$ of the collection of the invariants  
  $\{\,I^k_{\E_1},\,\bigl(\xi_j(I^k)\bigr)_{\E_1}\,\}_{k=1,2,\ldots,6,\,j=1,2}$ are functionally independent in some 
  neighborhood of $p_1$. Then there exists a point transformation of neighborhoods of $p_1$ and $p_2$ 
  transforming $\E_1$ to $\E_2$ and taking $p_1$ to $p_2$ iff the following conditions hold:
  \begin{enumerate}
  \item The invariants $J^1_{\E_2},\,J^2_{\E_2}$ are functionally independent in some neighborhood of
    $p_2$, $J^1_{\E_1}(p_1)=J^1_{\E_2}(p_2)$, and $J^2_{\E_1}(p_1)=J^2_{\E_2}(p_2)$.
  \item If $I^k_{\E_1}=F^k(J^1_{\E_1}, J^2_{\E_1})$ and $\bigl(\xi_j(I^k)\bigr)_{\E_1}=F^k_j(J^1_{\E_1}, J^2_{\E_1})$ in 
    some neighborhood of $p_1$, then $I^k_{\E_2}=F^k(J^1_{\E_2}, J^2_{\E_2})$ and
    $\bigl(\xi_j(I^k)\bigr)_{\E_2}=F^k_j(J^1_{\E_2}, J^2_{\E_2})$ in some
    neighborhood of $p_2$, $k=1,2,\ldots,6$, $j=1,2$.
  \end{enumerate}
\end{theorem}

\section{Appendix}
\subsection{The proof of Proposition \ref{LieAlgIsm}}\label{SctnLieAlgIsm}
Suppose $X$ and $Y$ are vector fields on the base of $\pi$, $f_t$ and $g_s$ are their flows respectively. Then
  \begin{multline*}
    [\,X^{(k)},\,Y^{(k)}\,]=\lim_{t\to 0}\frac{1}{t}
    \Bigl(\,Y^{(k)}-(f^{(k)}_t)_*(Y^{(k)}\circ
    f^{(k)}_{-t})\,\Bigr)\\
    =\lim_{t\to 0}\frac{1}{t}
    \Bigl(\,\frac{d}{ds}\Bigl|_{s=0}g^{(k)}_s
    -(f^{(k)}_t)_*\bigl(\frac{d}{ds}\Bigl|_{s=0}g^{(k)}_s
    \circ f^{(k)}_{-t}\bigr)\,\Bigr)
    =\lim_{t\to 0}\frac{1}{t}
    \Bigl(\,\frac{d}{ds}\Bigl|_{s=0}g^{(k)}_s\\
    -\frac{d}{ds}\Bigl|_{s=0}f^{(k)}_t\circ\,g^{(k)}_s
    \circ f^{(k)}_{-t}\,\Bigr)
    =\lim_{t\to 0}\frac{1}{t}\frac{d}{ds}\Bigl|_{s=0}
    \Bigl(\,g^{(k)}_s\circ f^{(k)}_t\circ\,g^{(k)}_s
    \circ f^{(k)}_{-t}\,\Bigr)\\
    =\lim_{t\to 0}\frac{1}{t}\frac{d}{ds}\Bigl|_{s=0}
    \Bigl(\,g_s\circ f_t\circ\,g_s
    \circ f_{-t}\,\Bigr)^{(k)}
    =\lim_{t\to 0}\frac{1}{t}
    \Bigl(\,\frac{d}{ds}\Bigl|_{s=0}g_s\\
    -\frac{d}{ds}\Bigl|_{s=0}f_t\circ\,g_s
    \circ f_{-t}\,\Bigr)^{(k)}
    =\lim_{t\to 0}\frac{1}{t}
    \bigl(\,Y-(f_t)_*(Y\circ f_{-t})\,\bigr)^{(k)}
    =[\,X,\,Y\,]^{(k)}\,.
  \end{multline*}
  The $\R$ -- linearity of the map $ X\mapsto X^{(k)}$ is obvious. This completes the proof.

\subsection{The proof of Theorem \ref{ThOmgIndH}}\label{PrfThOmgIndH}
Let $0_1$ be the point of $J^1\pi$ such that its standard coordinates are zeros and let $\theta_2$ be an arbitrary point of the fiber $(\pi_{2,1})^{-1}(0_1)$. Taking into account that $J^1\pi$ is an orbit of the action of $\Gamma$, we obtain that it is enough to prove this theorem for the point $\theta_2$.

From propositions \ref{IstrpAlgkPr}, we get that the algebra $\g_{0_1}$ is defined by the system
\begin{align*}
  X^2_{11}&=0, & X^1_{11} - 2X^2_{12}&=0, & 2X^1_{12} - X^2_{22}&=0, & X^1_{22}&=0,\\
  X^2_{111}&=0,& X^1_{111} - 2X^2_{112}&=0,& 2X^1_{112} - X^2_{122}&=0,& X^1_{122}&=0,\\
  X^2_{112}&=0,& X^1_{112} - 2X^2_{122}&=0,& 2X^1_{122} - X^2_{222}&=0,& X^1_{222}&=0.
\end{align*}
Whence,  
$$
  \g_{0_1}=\bigl\{\;j^3_0X=(\,0,\;X^i_j,\;X^i_{j_1j_2},\;0\,)\;
  \bigr\}\,,
$$
where the components $X^i_j$ are arbitrary and the components $X^i_{j_1j_2}$ satisfy to \eqref{g2}. From propositions \ref{IstrpSpcDscrb}, we get that the space $\A_{\theta_2}$ is defined by the system
\begin{gather*}
X^2_{11}=0,\quad   X^1_{11}-2X^2_{12}=0,\quad   2X^1_{12} - X^2_{22}=0,\quad   X^1_{22}=0,\\
X^2_{111}=u^1_{1i}X^i,\quad -X^1_{111}+2X^2_{112}=u^2_{1i}X^i,\quad  -2X^1_{112}+X^2_{122}=u^3_{1i}X^i,\\ X^1_{122}=-u^4_{1i}X^i,\\
X^2_{112}=u^1_{i2}X^i,\quad -X^1_{112}+2X^2_{122}=u^2_{i2}X^i,\quad -2X^1_{122}+X^2_{222}=u^3_{i2}X^i,\\ X^1_{222}=-u^4_{i2}X^i.
\end{gather*}
Whence, 
$$
  \A_{\theta_2}=\bigl\{\;j^3_0X=(\,X^i,\;X^i_j,\;X^i_{j_1j_2},\;
  X^i_{j_1j_2j_3}\,)\;\bigr\}\,,
$$
where the components $X^i_j$ are arbitrary and the components $X^i_{j_1j_2}$ satisfy to \eqref{g2}. It follows that there exists a horizontal subspace $H\subset\A_{\theta_2}$ such that
$$
  H=\bigl\{\;j^3_0X=(\,X^i,\;0,\;0,\; h^i_{j_1j_2j_3,r}X^r\,)\;\bigr\}\,.
$$
Obviously, this subspace satisfies condition \eqref{HS1}. Hence 
$$
  \omega_{H}(X_p, Y_p)=\bigl(\,(h^i_{j_1j_2r,s}-h^i_{j_1j_2s,r})X^rY^s\,\bigr)\,,
$$  
where $X_p=(X^1,X^2)$ and $Y_p=(Y^1,Y^2)$.
An arbitrary horizontal subspace $\tilde H\subset\A_{\theta_2}$ has the form 
$$
  \tilde H=\bigl\{\;j^3_0X=(\,X^i,\;h^i_{j,r}X^r,\;h^i_{j_1j_2,r}X^r,\;h^i_{j_1j_2j_3,r}X^r\,)\;\bigr\}\,,
$$
where the components $h^i_{j_1j_2,r}X^r$ satisfy to \eqref{g2}. Let $\tilde H$ satisfies \eqref{HS1}. Then
\begin{multline*}
  \omega_{\tilde H}(X_p, Y_p)=\bigl(\,(\,h^k_{j_1j_2,r}h^i_{k,s}-h^k_{j_1j_2,s}h^i_{k,r}+h^k_{j_1,r}h^i_{kj_2,s}
  -h^k_{j_1,s}h^i_{kj_2,r}\\
  +h^k_{j_2,r}h^i_{kj_1,s}-h^k_{j_2,s}h^i_{kj_1,r}+h^i_{j_1j_2r,s}-h^i_{j_1j_2s,r}\,)X^rY^s\,\bigr)\,.
\end{multline*}
Consequently, 
\begin{multline*}
  (\,\omega_{\tilde H}-\omega_{H}\,)(X_p,Y_p)\\
  =(\,h^k_{j_1j_2,r}h^i_{k,s}-h^k_{j_1j_2,s}h^i_{k,r}
  +h^k_{j_1,r}h^i_{kj_2,s}
  -h^k_{j_1,s}h^i_{kj_2,r}\\
  +h^k_{j_2,r}h^i_{kj_1,s}-h^k_{j_2,s}h^i_{kj_1,r}\,)X^rY^s\,.
\end{multline*}
Prove that $\omega_{\tilde H}-\omega_{H}=0$. Condition \eqref{HS1} for $\tilde H$ means that
$$
  \begin{aligned}
    h^i_{s,r}&=h^i_{r,s}\;\;\forall\;i,r,s\quad\text{and}\\
    h^1_{11,2}-h^1_{12,1}&=h^2_{1,2}h^1_{1,2}-h^2_{1,1}h^1_{2,2}\,,\\
    h^1_{12,2}-h^1_{22,1}&=h^1_{2,2}h^1_{1,1}+h^2_{2,2}h^1_{1,2}
                         -h^1_{1,2}h^1_{1,2}-h^2_{1,2}h^1_{2,2}\,,\\
    h^2_{11,2}-h^2_{12,1}&=h^1_{1,2}h^2_{1,1}+h^2_{1,2}h^2_{1,2}
                         -h^1_{1,1}h^2_{1,2}-h^2_{1,1}h^2_{2,2}\,,\\
    h^2_{12,2}-h^2_{22,1}&=h^1_{2,2}h^2_{1,1}-h^1_{1,2}h^2_{1,2}\,.
  \end{aligned}
$$
Taking into account that the components $h^i_{j_1j_2,r}X^r$ satisfy
\eqref{g2}, we can rewrite the last system in the following way
$$
  \begin{aligned}
    2h^2_{12,2}-h^1_{12,1}&=h^2_{1,2}h^1_{1,2}-h^2_{1,1}h^1_{2,2}\,,\\
    h^1_{12,2}&=h^1_{2,2}h^1_{1,1}+h^2_{2,2}h^1_{1,2}
                         -h^1_{1,2}h^1_{1,2}-h^2_{1,2}h^1_{2,2}\,,\\
    -h^2_{12,1}&=h^1_{1,2}h^2_{1,1}+h^2_{1,2}h^2_{1,2}
                         -h^1_{1,1}h^2_{1,2}-h^2_{1,1}h^2_{2,2}\,,\\
    h^2_{12,2}-2h^1_{12,1}&=h^1_{2,2}h^2_{1,1}-h^1_{1,2}h^2_{1,2}\,.
  \end{aligned}
$$
From this system, the components $h^i_{j_1j_2,k}$ are expressed in the terms of the components $h^i_{j,k}$ in the following way: 
\begin{equation}\label{SH1spc}
  \begin{aligned}
    h^2_{12,2}&=h^2_{1,2}h^1_{1,2}-h^2_{1,1}h^1_{2,2}\,,\\
    h^1_{12,2}&=h^1_{2,2}h^1_{1,1}+h^2_{2,2}h^1_{1,2}
                         -h^1_{1,2}h^1_{1,2}-h^2_{1,2}h^1_{2,2}\,,\\
    h^2_{12,1}&=-h^1_{1,2}h^2_{1,1}-h^2_{1,2}h^2_{1,2}
                         +h^1_{1,1}h^2_{1,2}+h^2_{1,1}h^2_{2,2}\,,\\
    h^1_{12,1}&=h^2_{1,2}h^1_{1,2}-h^2_{1,1}h^1_{2,2}\,.
  \end{aligned}
\end{equation}
The values of the 2--form $\omega_{\tilde H}-\omega_{H}$ belong to $g^2$ and $g^2$ is defined by \eqref{g2}. Therefore to prove that $\omega_{\tilde H}-\omega_{H}$ is zero, it is enough to check that the components $(\omega_{\tilde H}-\omega_{H})^1_{12,12}$ and $(\omega_{\tilde H}-\omega_{H})^2_{12,12}$ of $\omega_{\tilde H}-\omega_{H}$ are zeros. It can be easily checked by direct calculations applying \eqref{SH1spc}, the equalities $h^i_{s,r}=h^i_{r,s}$, and taking into account that the components $h^i_{j_1j_2,r}X^r$ satisfy \eqref{g2}. This concludes the proof.

\subsection{The expression of ${\bf \omega^2}$ in the standard coordinates}\label{SctnOmg2Expr} 

Let $\theta_2$ be an arbitrary point of $J^2\pi$, $\theta_1=\pi_{2,1}(\theta_2)$, and $p=\pi_2(\theta_2)$. Then it follows from $g^1_{\theta_1}=L_p^0/L_p^1$ and $L_p^1/L_p^2\subset L_p^0/L_p^1\otimes T_p^*$ that there are horizontal subspaces of $H\subset\A_{\theta_2}$ so that 
\begin{equation}\label{HS2}
  H=\bigl\{\;j^3_pX=(\,X^i,\;0,\;h^i_{j_1j_2,r}X^r,\; h^i_{j_1j_2j_3,r}X^r\,)\,\bigr\}\,.
\end{equation}
Suppose $H$ is one of these horizontal subspaces and $j_p^3X, j_p^3Y\in H$. Then 
$$
 [\,j_p^3X, j_p^3Y\,]=\bigl(\,0\,,\;(h^i_{j_1r,s}-h^i_{j_1s,r})X^rY^s\,,\;
 (h^i_{j_1j_2r,s}-h^i_{j_1j_2s,r})X^rY^s\\,\bigr)\,. 
$$
It follows that $H$ satisfies \eqref{HS1} iff 
$$
  h^i_{jr,s}=h^i_{jr,s}\quad\forall\,i,j,r,s\,.
$$
Vectors $j^3_pX\in H$ satisfy the system, 
\begin{gather}
    \psi^i_{X}(\theta_1)=0\,,\quad D_1(\,\psi^i_{X}\,)(\theta_2)=0\,,\quad            
    D_2(\,\psi^i_{X}\,)(\theta_2)=0\,,\label{IstrpSpc2}\\
     i=1,2,3,4\notag
\end{gather}
defining $\A_{\theta_2}$, see Proposition \ref{IstrpSpcDscrb}. The eight last equations of this system express the components $h^i_{j_1j_2j_3,r}$ in the terms of $h^i_{j_1j_2,r}$. The four first equations of this system 
\begin{align*}
    -u^0_1X^1-u^0_2X^2 + h^2_{11,r}X^r&=0\,,\\
    -u^1_1X^1-u^1_2X^2 - h^1_{11,r}X^r + 2h^2_{12,r}X^r&=0\,,\\
    -u^2_1X^1-u^2_2X^2 - 2h^1_{12,r}X^r + h^2_{22,r}X^r&=0\,,\\
    -u^3_1X^1-u^3_2X^2 - h^1_{22,r}X^r&=0
\end{align*}
connect the components $h^i_{j_1j_2,r}$. It is easy to see that this system has a unique solution $h^i_{j_1j_2,r}$ satisfying the condition $h^i_{jr,s}=h^i_{jr,s}$ for all $i,j,r,s$. This solution is described by the formulas
$$
 \begin{aligned}
    h^1_{11,1}=2u^0_2-u^1_1\,,\;h^1_{11,2}=\frac{1}{3}(\,u^1_2-2u^2_1\,)\,,\;h^1_{12,2}=-u^3_1\,,\;h^1_{22,2}=-u^3_2\,,\\
    h^2_{11,1}=u^0_1\,,\;h^2_{11,2}=u^0_2\,,\;h^2_{12,2}=\frac{1}{3}(\,2u^1_2-u^0_1\,)\,,\;h^2_{22,2}=-2u^3_1+u^2_2\,.   
 \end{aligned}
$$
Thus in $\A_{\theta_2}$, there exists a unique horizontal subspace $H$ satisfying \eqref{HS2} and \eqref{HS1}. 
From above-mentioned formula for brackets of vectors of $H$, we get 
$$
  \omega_{\theta_2}=(h^i_{j_1j_2r,s}-h^i_{j_1j_2s,r})\bigl(\frac{\partial}{\partial x^i}
  \otimes (dx^{j_1}\odot dx^{j_2})\bigr)\otimes\,(dx^r\wedge\,dx^s)\,.
$$
Taking into account \eqref{g2} and \eqref{Gnrtrs_g2}, we get
$$
  \omega_{\theta_2}=\bigl((h^1_{111,2}-h^1_{112,1})\cdot e_1+(h^2_{221,2}-h^2_{222,1})\cdot e_2\bigr)
  \otimes\,(dx^1\wedge\,dx^2)\,.
$$
From the eight last equations of system \eqref{IstrpSpc2}, we get that 
$$
  h^1_{111,2}-h^1_{112,1}=F^1\,,\quad h^2_{221,2}-h^2_{222,1}=F^2\,,
$$
where $F^1$ and $F^2$ are defined by \eqref{F1F2}. Thus we obtain the following expression of $\omega^2$ in the
standard coordinates
\begin{multline*}
  \omega^2=\Bigl(F^1\bigl(2\frac{\partial}{\partial x^1}
  \otimes (dx^1\odot dx^1)
  +\frac{\partial}{\partial x^2}\otimes (dx^1\odot dx^2)\bigr)\\
  +F^2\bigl(2\frac{\partial}{\partial x^2}\otimes (dx^2\odot dx^2)
  +\frac{\partial}{\partial x^1}\otimes
  (dx^1\odot dx^2)\bigr)\Bigr)\\
  \otimes\,(dx^1\wedge\,dx^2)\,.
\end{multline*}

\subsection{The proof of Lemma \ref{Lmm}}\label{PrfLmm}
Let us calculate components of the tensor $t_H$ in the standard base of $T_p$ to check the required symmetry.

In the standard coordinates, the horizontal subspace $H$ is defined by the quantities $h^i_{j,k}$, $h^i_{j_1j_2,k}$, $\ldots$, $h^i_{j_1\ldots j_4,k}$, $i,j_1,\ldots,j_4,k=1,2$
$$
  H=\bigl\{\,j^4_pX=\bigl(\,X^i,\; X^k(\,h^i_{j,k},\,h^i_{j_1j_2,k},\,h^i_{j_1j_2j_3,k},\,
  h^i_{j_1j_2j_3j_4,k}\,)\,\bigr).
$$
Condition \eqref{HS03} means
$$
  h^i_{j,k}=h^i_{k,j}\quad\forall\,i,j,k.
$$
Let  
\begin{equation}\label{BrHr3}
  [j_p^4X, j_p^4Y]=(\,0,\;0,\;g^i_{j_1j_2},\; g^i_{j_1j_2,j_3}\,)
\end{equation}
in the standard coordinates.
Then
\begin{multline*}
 w=\bigl[ j_p^3Z,[j_p^4X, j_p^4Y]\bigr]\\
 =Z^k(\,g^i_{jk},\;-g^r_{j_1j_2}h^i_{r,k}+h^r_{j_1,k}g^i_{rj_2}+h^r_{j_2,k}g^i_{rj_1}
       +g^i_{j_2j_1k} \,),
\end{multline*}
\begin{multline*}
  [j_p^2U,w]=U^mZ^k(\,g^i_{mk},\;
   h^r_{j,m}g^i_{rk}-g^r_{jk}h^i_{r,m}\\
   -g^r_{jm}h^i_{r,k}+h^r_{j,k}g^i_{rm}+h^r_{m,k}g^i_{rj}+g^i_{jmk}\,)
\end{multline*}
and
\begin{multline}\label{omg03}
  [j_p^2U,w]-j_p^1\tilde Z=U^mZ^k(\,h^r_{j,m}g^i_{rk}-g^r_{jk}h^i_{r,m}-h^i_{j,r}g^r_{mk}\\
   -g^r_{jm}h^i_{r,k}+h^r_{j,k}g^i_{rm}+h^r_{m,k}g^i_{rj}+g^i_{jmk}\,).
\end{multline}
Taking into account that $h^i_{j,k}=h^i_{k,j}$ for all $i,j,k$, we get that the expression
$$
  h^r_{j,m}g^i_{rk}-g^r_{jk}h^i_{r,m}-g^r_{mk}h^i_{r,j}
   -g^r_{jm}h^i_{r,k}+h^r_{j,k}g^i_{rm}+h^r_{m,k}g^i_{rj}+g^i_{jmk}.
$$
is symmetric over the indexes $j$, $k$ and $m$. This completes the proof.

\subsection{The proof of Theorem \ref{ThrmOmg3H}}\label{SctnThrmOmg3H}
We calculate the tensor $\omega_H^3$ in the standard coordinates to prove the theorem. To this end, we carry out the following steps: 
\begin{enumerate}
 \item Calculate components $h^i_{j,k}$ of a horizontal subspace $H$ satisfying property \eqref{HS3}.
 \item Calculate the components $g^i_{j_1j_2}$ and $g^i_{j_1j_2,j_3}$ of $[j_p^4X, j_p^4Y]$, where\linebreak
  $j_p^4X, j_p^4Y\in H$.
 \item Calculate the components $t^i_{jmk}$ of the tensors $t_H$ and the components $\tilde t^i_{jmk}$ of the    
  tensors $\tilde\mu(t_H)$.
 \item Finally, calculate components of the tensor $\omega_H^3$.
\end{enumerate}

Recall that $\theta_3\in(\pi_{3,2})^{-1}(\Orb_2^0)$. Therefore $\bigl(F^1(\theta_2),F^2(\theta_2)\bigr)\neq 0$ for $\theta_2=\pi_{3,2}(\theta_3)$. 

In this proof, we suppose that $F^1(\theta_2)\neq 0$. For the case $F^2(\theta_2)\neq 0$, the proof is the same. We omit it. 

Step (1). Let $H$ be an arbitrary horizontal subspace of $\A_{\theta_3}$ satisfying \eqref{HS3}. Then it follows from \eqref{IstrSpc3} that the components $h^i_{j,k}$ of $H$ are defined by the system of equations
$$
  \begin{aligned}
     &h^i_{j,k}=h^i_{k,j}\quad\forall\, i,j,k,\\
     &2F^1h^1_{1,1}+F^2h^2_{1,1}+F^1h^2_{1,2}=-D_1F^1,\\
     &2F^1h^1_{1,2}+F^2h^2_{1,2}+F^1h^2_{2,2}=-D_2F^1,\\
     &F^2h^1_{1,1}+F^1h^1_{1,2}+2F^2h^2_{1,2}=-D_1F^2,\\
     &F^2h^1_{1,2}+F^1h^1_{2,2}+2F^2h^2_{2,2}=-D_2F^2.
  \end{aligned}
$$
From this system, we get
\begin{equation}\label{HS13}
 \begin{aligned}
  h^1_{1,2}&=h^1_{2,1}=\bigl(2h^2_{1,1}(F^2)^2 + 3h^1_{1,1}F^1F^2 - F^1F^2_x + 2F^1_xF^2\bigr)/(F^1)^2,\\
  h^1_{2,2}&=\bigl(4h^2_{1,1}(F^2)^3 + 5h^1_{1,1}F^1(F^2)^2\\ 
           &\quad - (F^1)^2F^2_y + 2F^1F^1_yF^2 - 3F^1F^2F^2_x + 4F^1_x(F^2)^2\bigr)/(F^1)^3,\\
  h^2_{1,2}&=h^2_{2,1}=\bigl( - h^2_{1,1}F^2 - 2h^1_{1,1}F^1 - F^1_x\bigr)/F^1,\\
  h^2_{2,2}&=\bigl( - 3h^2_{1,1}(F^2)^2 - 4h^1_{1,1}F^1F^2\\ 
           &\quad - F^1F^1_y + 2F^1F^2_x - 3F^1_xF^2\bigr)/(F^1)^2.
 \end{aligned}
\end{equation}
Here the components $h^1_{1,1}$ and $h^2_{1,1}$ are arbitrary. They define the arbitrariness in the choice of a horizontal subspace $H$ satisfying property \eqref{HS3}.

Step (2). It is clear from the construction of the invariant $\omega^2$ that the components $g^i_{jk}$ in \eqref{BrHr3}, which are symmetric over the indexes $j$ and $k$, are defined by
\begin{equation}\label{gijk}
 \begin{aligned} 
  g^1_{12}&=F^2(\theta_2)\lambda,\quad g^2_{12}=F^1(\theta_2)\lambda,\quad\lambda=X^1Y^2-X^2Y^1,\\
  g^1_{11}&=2g^2_{12},\quad g^2_{22}=2g^1_{12},\quad g^1_{22}=0,\quad g^2_{11}=0.
 \end{aligned}
\end{equation}
It follows from \eqref{FstPrlng1} and the system of equation defining $\g_{\theta_1}$, see Proposition \ref{IstrpAlgkPr}, that the components $g^i_{j_1j_2j_3}$ in \eqref{BrHr3}, which are symmetric over the indexes $j_1$, $j_2$, and $j_3$, are defined by the equations
\begin{equation}\label{gij1j2j3}
 \begin{aligned}
  &g^1_{111}=-3F^2u^0\lambda,\;g^1_{112}=-F^2u^1\lambda,\;g^1_{122}=-F^2u^2\lambda,\;g^1_{222}=-3F^2u^3\lambda,\\
  &g^2_{111}=3F^1u^0\lambda,\phantom{-}\;g^2_{112}=F^1u^1\lambda,\phantom{-}\;g^2_{122}=F^1u^2\lambda,\phantom{-}\;   
  g^2_{222}=3F^1u^3\lambda\,.
 \end{aligned}
\end{equation}

Step (3). Substituting \eqref{HS13}, \eqref{gijk} and \eqref{gij1j2j3} in \eqref{omg03}, we obtain the components $t^i_{jmk}$ of the tensor $t_H$:
\begin{align*}
t^1_{111}&=3\lambda F^2(h^2_{11} - u^0),\\ 
t^1_{112}&=t^1_{121}=t^1_{211}=\lambda F^2( - 2h^2_{11}F^2 - 5h^1_{11}F^1 - F^1u^1 - 2F^1_x)/F^1,\\ 
t^1_{122}&=t^1_{212}=t^1_{221}=\lambda F^2\bigl( - 7h^2_{11}(F^2)^2 - 10h^1_{11}F^1F^2 - (F^1)^2u^2 - F^1F^1_y\\ 
         &\phantom{=t^1_{212}=t^1_{221}=\lambda F^2\bigl( }+ 4F^1F^2_x - 7F^1_xF^2\bigr)/\bigl((F^1)^2\bigr),\\
t^1_{222}&=3\lambda F^2\bigl( - 4h^2_{11}(F^2)^3 - 5h^1_{11}F^1(F^2)^2 - (F^1)^3u^3 + (F^1)^2F^2_y\\ 
         &\phantom{=3\lambda F^2\bigl( }- 2F^1F^1_yF^2 + 3F^1F^2F^2_x - 4F^1_x(F^2)^2\bigr)/\bigl((F^1)^3\bigr),\\
t^2_{111}&=3\lambda F^1( - h^2_{11} + u^0),\\
t^2_{112}&=t^2_{121}=t^2_{211}=\lambda (2h^2_{11}F^2 + 5h^1_{11}F^1 + F^1u^1 + 2F^1_x),\\
t^2_{122}&=t^2_{212}=t^2_{221}=\lambda \bigl(7h^2_{11}(F^2)^2 + 10h^1_{11}F^1F^2 + (F^1)^2u^2 + F^1F^1_y\\ 
         &\phantom{t^2_{212}=t^2_{221}=\lambda \bigl(}- 4F^1F^2_x + 7F^1_xF^2\bigr)/F^1,\\
t^2_{222}&=3\lambda\bigl(4h^2_{11}(F^2)^3 + 5h^1_{11}F^1(F^2)^2 + (F^1)^3u^3 - (F^1)^2F^2_y\\ 
         &\phantom{=3\lambda\bigl(} + 2F^1F^1_yF^2 - 3F^1F^2F^2_x + 4F^1_x(F^2)^2\bigr)/\bigl((F^1)^2\bigr).
\end{align*}
Now we obtain the components $\tilde t^i_{jmk}$ of the tensors $\tilde\mu(t_H)$:
\begin{align*}
\tilde t^1_{121}&=\lambda\bigl(5h^2_{11}(F^2)^2 + 5h^1_{11}F^1F^2 + (F^1)^2u^2 + F^1F^1_y - F^1F^2u^1\\
                &\phantom{=\lambda\bigl(}- 4F^1F^2_x + 5F^1_xF^2\bigr)/(3F^1),\\
\tilde t^1_{122}&=\lambda\bigl(5h^2_{11}(F^2)^3 + 5h^1_{11}F^1(F^2)^2 + 3(F^1)^3u^3 - (F^1)^2F^2u^2 - 3(F^1)^2F^2_y\\ 
                &\phantom{=\lambda\bigl(}+ 5F^1F^1_yF^2 - 5F^1F^2F^2_x + 5F^1_x(F^2)^2\bigr)/\bigl(3(F^1)^2\bigr),\\
\tilde t^2_{121}&=\lambda(5h^2_{11}F^2 + 5h^1_{11}F^1 + F^1u^1 + 2F^1_x - 3F^2u^0)/3,\\
\tilde t^2_{122}&=\lambda\bigl(5h^2_{11}(F^2)^2 + 5h^1_{11}F^1F^2 + (F^1)^2u^2 + F^1F^1_y - F^1F^2u^1\\ 
                &\phantom{=\lambda\bigl(}- 4F^1F^2_x + 5F^1_xF^2\bigr)/(3F^1),\\
\tilde t^1_{121}&=\tilde t^1_{211},\quad\tilde t^1_{122}=\tilde t^1_{212},\quad
\tilde t^2_{121}=\tilde t^2_{211},\quad\tilde t^2_{122}=\tilde t^2_{212},\\
\tilde t^1_{111}&=2\tilde t^2_{121},\quad\tilde t^1_{112}=2\tilde t^2_{122},\quad
\tilde t^2_{221}=2\tilde t^1_{121},\quad\tilde t^2_{222}=2\tilde t^1_{122},\\
\tilde t^1_{221}&=\tilde t^1_{222}=\tilde t^2_{111}=\tilde t^2_{112}=0.
\end{align*}

Step (4). Taking into account expression \eqref{beta2} of $\beta_{\theta_2}$, we obtain the components $(\omega^3_H)^i_{jk}$ 
of the tensor $\omega^3_H=\beta_{\theta_2}\kch\tilde\mu(t_H)$:
\begin{align*}
(\omega^3_H)^1_{12}&=\lambda^3\bigl( - 3(F^1)^2u^3 + 2F^1F^2u^2 - (F^2)^2u^1\\ 
                   &\phantom{=\lambda^3\bigl( - 3(F^1)^2u^3 + 2F^1F^2u^2 }+ 3F^1F^2_y - 4F^1_yF^2 + F^2F^2_x\bigr)/3,\\
(\omega^3_H)^2_{12}&=\lambda^3\bigl( - (F^1)^2u^2 + 2F^1F^2u^1 - 3(F^2)^2u^0\\ 
                   &\phantom{=\lambda^3\bigl( - (F^1)^2u^2 + 2F^1F^2u^1 }- F^1F^1_y + 4F^1F^2_x - 3F^1_xF^2\bigr)/3,\\
(\omega^3_H)^1_{11}&=2(\omega^3_H)^2_{12},\quad(\omega^3_H)^2_{22}=2(\omega^3_H)^1_{12},\quad
(\omega^3_H)^1_{22}=(\omega^3_H)^2_{11}=0.
\end{align*}
We get that the components of $\omega^3_H$ are independent of the choice of a horizontal subspace $H$ satisfying \eqref{HS3}. This completes the proof.


%

\begin{thebibliography}{99}
\bibitem{ALV} D.V.Alekseevskiy, V.V.Lychagin, A.M.Vinogradov, {\it Fundamental ideas and conceptions of differential 
  geometry, Sovremennye problemy matematiki. Fundamental'nye napravleniy}, Vol. 28 (Itogi nauki i techniki, VINITI, AN  
  SSSR, Moscow, 1988 (Russian)) [English transl.: Encyclopedia of Math. Sciences, Vol.28 (Springer, Berlin, 1991)]
\bibitem{Arnld} V.I.Arnold, {\it Advanced chapters of the theory of ordinary differential equations}, Nauka, Moskow, 
  1978 (in Russian).
\bibitem{BbchBrdg} M.V. Babich and L.A. Bordag, {\it Projective differential geometrical structure of the Painleve 
  equations}, J. Dif. Equations, 1999, V.157, No.2, pp.452-485.
\bibitem{BrnshtnRznfld} I.N.Bernshteyn, B.I.Rozenfel'd, {\it Homogeneous spaces of infinite dimensional Lie algebras and
  characteristic classes of foliations}, Uspekhi Matematicheskikh Nauk, vol 28, No. 4, pp. 103-138, 1973. (in Russian)
\bibitem{Chrn} S.S. Chern, {\it Projective geometry, contact transformations, and $CR$-structures},
  1982, Arch. Math. Vol. 38, pp. 1 - 5.
\bibitem{Crtn} E. Cartan, {\it Sur les varietes a connexion projective}, Bull. Soc. Math. France 52 (1924), 205 -- 241.
\bibitem{Grdnr} R.B. Gardner, {\it Differential geometric methods interacting control theory}, 
  in "Differential geometry control theory" (R.W. Brockett et al., Eds.), pp.117-180, Birkhauser, Boston, 1983.
\bibitem{GrssmThmpsnWlkns} C. Grissom, G. Thompson, and G. Wilkens, {\it Linearization of second order ordinary differential    equations via Cartan's equivalence method}, J. Differential Equations, 1989, Vol. 77, pp.1--15.
\bibitem{GllmnStrnbrg1} V.Guillemin, S.Sternberg, {\it An algebraic model of transitive differential geometry}, Bull. 
  Amer. Math. Soc., vol. 70, (1964), pp. 16-47.
\bibitem{GllmnStrnbrg2} V.Guillemin, S.Sternberg, {\it Deformation theory of pseudogroup structures}, 
  AMS, No. 64, 1966, pp. 1-80.
\bibitem{GYum} V.N.Gusyatnikova, V.A.Yumaguzhin, {\it Point transformations and linearisability of 2-order ordinary
  differential equations}, Matemeticheskie Zametki Vol. 49, No. 1,
  pp. 146 - 148, 1991 (in Russian).
\bibitem{HsKmrn} L. Hsu and N. Kamran, {\it Classification of second-order ordinary 
  differential equations admitting Lie groups of fiber-preserving point 
  symmetries}, Proc. London Math. Soc., (3), 58(1989), 387-416
\bibitem{KmLmSh} N. Kamran, K.G. Lamb, and W.F. Shadwick, {\it The local equivalence problem for  
  $d^2y/dx^2=F(x,y,dy/dx)$ and the Painleve transcendents}, J. Dif. Geometry, 22, 1985, 139-150.
\bibitem{KLV} I.S.Krasil'shchik, V.V.Lychagin, A.M.Vinogradov, {\it Geometry of Jet Spaces and Nonlinear Partial   
  Differential Equations}, Gordon and Breach, New York, 1986.
\bibitem{KV} I.S. Krasil'shchik and A.M. Vinogradov, Editors, {\it Symmetries and conservation laws for differential    
  equations of mathematical Physics}, Translations of Mathematical Monographs. Vol.182, Providence RI: American   
  Mathematical Society, 1999.
\bibitem{Krnsh} M.Kuranishi, {\it Lectures on involutive systems of partial differential equations}, S\~ao Paulo, 1967.
\bibitem{Li1} S. Lie, {\it Vorlesungen \"uber continuierliche gruppen}, Teubner, Leipzig, 1893.
\bibitem{Li2} S. Lie, {\it Theorie der transformationsgruppen}, Vol. III, Teubner, Leipzig, 1930.
\bibitem{Lvll} R. Liouville, {\it Sur les invariantes de certaines equationes differentielles}, 
   Jour. de l'Ecole Politechnique, 59 (1889) 7--88.
\bibitem{SngrStrnbrg} I.M. Singer, S. Sternberg, {\it On the infinite groups of Lie and Cartan,I}, 
  J. Analyse Math. Vol. 15,  pp. 1-114, 1965.
\bibitem{Strnbrg} S. Sternberg, {\it Lectures on Differential Geometry}, New Jersy, Prentice-Hall, Inc., 1964.
\bibitem{Thmsn} G. Thomsen, {\it Uber die topologischen Invarianten der Dif\-fe\-ren\-ti\-alg\-lei\-chung
  $y''=f(x,y)y'\,^3+g(x,y)y'\,^2+h(x,y)y'+k(x,y)$}, Abh. Math. Sem. Hamburg. Univ. 7 (1930), 301-328.
\bibitem{Trss} A. Tresse, {\it Sur les invariants differentiels des groupes continus de transformations}, 
  Acta Math. 18 (1894), 1-88.
\bibitem{Thmpsn} G. Thompson, {\it Cartan's method of equivalence and second-order equation fields}, Letter to the  
  editor, J. Phys. A.: Math. Gen. 18 (1985), L1009-L1015.
\bibitem{Vngrdv} A.M. Vinogradov, {\it Scalar differential invariants, diffieties and characteristic classes}, in: 
  Mechanics, Analysis and Geometry: 200 Years after Lagrange, ed. M. Francaviglia (North-Holland), pp.379--414, 1991.
\bibitem{Yum} V.A.Yumaguzhin, {\it On the obstruction to linearizability of 2-order ordinary differential equations},
  Acta Applicandae Mathematicae, Vol. 83, No. 1-2, 2004. pp.133-148. 
  arXiv:0804.0306
\end{thebibliography}
\end{document}